\documentclass[final,leqno]{siamltex704}
\usepackage{amsmath}
\usepackage{graphicx}
 \usepackage[notcite,notref]{showkeys}
\usepackage{mathrsfs}
\usepackage{color}
\usepackage[dvipsnames]{xcolor}
\usepackage{float}
\usepackage{amsfonts,amssymb}
\usepackage{dsfont}
\usepackage{pifont}
\usepackage{wrapfig} 
\usepackage{hyperref}
\usepackage{multirow}
\numberwithin{equation}{section}
\def\3bar{{|\hspace{-.02in}|\hspace{-.02in}|}}
\def\E{{\mathcal{E}}}
\def\T{{\mathcal{T}}}

\def\pT{{\partial T}}

\def\W{{\mathcal{W}}}

\def\bn{{\mathbf{n}}}

\def\bx{{\mathbf{x}}}

\def\bF{{\textbf{F}}}
\def\ljump{{[\![}}
\def\rjump{{]\!]}}

\def\bbeta{{\boldsymbol{\beta}}}

\newtheorem{algorithm}{$L^p$-Primal-Dual Weak Galerkin Algorithm}[section]

\def\sgn{\operatornamewithlimits{sgn}}
\def\3bar{{|\hspace{-.02in}|\hspace{-.02in}|}}
\title{An $L^p$-primal-dual finite element method for first-order transport problems}

\author{
Dan Li\thanks{School of Mathematical Sciences,  Nanjing Normal University, Nanjing
210023, China (danlimath@163.com).} \and
Chunmei Wang \thanks{Department of Mathematics, University of Florida, Gainesville, FL 32611 (chunmei.wang@ufl.edu).  The research of Chunmei Wang was partially supported by National Science Foundation Grants DMS-2136380 and DMS-2206332.}
\and
Junping Wang\thanks{Division of Mathematical Sciences, National Science Foundation, Alexandria, VA 22314 (jwang@nsf.gov). The research of Junping Wang was supported by the NSF IR/D program, while working at National Science Foundation. However, any opinion, finding, and conclusions or recommendations expressed in this material are those of the author and do not necessarily reflect the views of the National Science Foundation.}}

\begin{document}

\maketitle
\begin{abstract}
A new $L^p$-primal-dual weak Galerkin method ($L^p$-PDWG) with $p>1$ is proposed for the first-order transport problems. The existence and uniqueness of the $L^p$-PDWG numerical solutions is established. In addition, the $L^p$-PDWG method offers a numerical solution which retains mass conservation locally on each element. An optimal order error estimate is established for the primal variable. A series of numerical results are presented to verify the efficiency and accuracy of the proposed $L^p$-PDWG scheme.
\end{abstract}



\pagestyle{myheadings}

\section{Introduction}
This paper is concerned with the new development of numerical methods for the first-order linear convection equations in divergence form by using a new $L^p$-primal-dual weak Galerkin method. For simplicity, we consider the model problem that seeks an unknown function $u$ satisfying
\begin{equation}\label{model}
\begin{split}
\nabla\cdot (\bbeta u)+c u=&f, \qquad \text{in }\Omega,\\
u=&g, \qquad  \text{on }\Gamma_-,
\end{split}
\end{equation}
where $\Omega \subset\mathbb R^d (d=2,3)$ is an open bounded and connected domain with Lipschitz continuous boundary $\partial\Omega$,   $\Gamma_-$ is the inflow portion of the boundary such that
$$
\Gamma_-=\{\bx\in\partial\Omega: \ \bbeta(\bx)\cdot\bn<0\},
$$
with $\bn$ being an unit outward normal vector along the boundary $\partial\Omega$. Assume that the convection vector  $\bbeta\in[L^{\infty}(\Omega)]^d$ is locally $W^{1,\infty}$, the reaction coefficient $c\in L^{\infty}(\Omega)$ is bounded, the load function $f\in L^q(\Omega)$ and the inflow boundary data $g\in L^q(\Gamma_-)$ are two given functions with $\frac{1}{p}+\frac{1}{q}=1$.

The first-order linear convection equations have  wide applications in science and engineering problems. For the non-smooth inflow boundary data, the exact solution of model problem \eqref{model} possesses discontinuity. It is well known that the solution discontinuity imposes a challenge on the development of efficient numerical methods. To address the challenge, a variety of numerical methods for linear convection equations have been developed such as streamline-upwind Petrov-Galerkin method \cite{BH1985, JNP1984}, stabilized finite element methods \cite{EB2014, JLG1999, RC2000, BH2004, BB2004}, least-square finite element method \cite{MY2018, CJ1988, BC2001, SMMO2005}, and discontinuous Galerkin finite element methods \cite{JP1986, HJS2002, BC1999, CS2001, EB2005, RH1973, EMS2004, EQS2010, BS2007}. Most of the aforementioned numerical methods assumed the coercivity assumption in the form of $c+\frac{1}{2}\nabla\cdot\bbeta\geq\alpha_0>0$ for a fixed constant $\alpha_0$ or alike. However, this assumption imposes a strong restriction in practice and often rules out some important physics such as exothermic reactions \cite{EB2014}. In \cite{EB2014}, the authors developed a stabilized finite element method for linear transport equation with smooth coefficients $\bbeta\in[W^{2,\infty}(\Omega)]^d$ and $c\in W^{1,\infty}(\Omega)$. In \cite{WW2020}, the primal-dual weak Galerkin method was proposed and analyzed for the linear transport equation with $\bbeta\in[W^{1,\infty}(\Omega)]^d$ and $c\in L^\infty(\Omega)$. In \cite{LWW2022}, the authors proposed a new primal-dual weak Galerkin method for transport equation in non-divergence form with locally  coefficients $\bbeta\in[C^{1,\alpha_0}]^d$ and $c\in C^{0,\alpha_0}$, where the PDWG method was first developed in \cite{ww2016} to solve many partial differential equations \cite{cww2022,ww2018,ww2020,wwec2020,w2020,wz2021} that are difficult to be approximated by using traditional numerical methods. The basic principle of primal-dual weak Galerkin method is to characterize this method as a linear $L^2$ minimization problem with constraints that satisfy the certain conditions weakly on each finite element.

The objective of this paper is to present an $L^p$-primal-dual weak Galerkin method for the linear convection problem \eqref{model} without enforcing the coercivity condition on the convection vector $\bbeta$ and the reaction coefficient $c$. The $L^p$-PDWG method has been successfully applied to solve some partial differential equations such as div-curl systems \cite{CWWdiv}, second order elliptic equations in non-divergence form \cite{CWWsecond}, convection-diffusion equations \cite{CWW2023}. Our numerical algorithm can be characterized as a constrained nonlinear $L^p$ optimization problem with constraints that satisfy the model equation weakly on each finite element by using locally designed weak gradient operator. Different from the numerical method in \cite{WW2020,LWW2022}, a more general $L^p$ stabilizer is introduced. The solvability and stability of $L^p$-PDWG method is discussed by introducing an equivalent min-max characterization and   an error estimate of optimal order is established for the numerical approximation.
In addition, a mass conservative numerical scheme is constructed to approximate the linear convection problem \eqref{model} where this numerical scheme is based on the $L^p$-primal-dual weak Galerkin technique.

This paper is organized as follows. In Section \ref{Section:DWG}, we introduce a weak formulation for the first-order transport problem \eqref{model} and then briefly review the definition of weak gradient operator. Section \ref{Section:WGFEM} presents an $L^p$-PDWG scheme for the model problem \eqref{model}. The goal of Section \ref{Section:EU} is to study the  solution existence and uniqueness for this numerical scheme. Section \ref{Section:MC} is devoted to discussing the mass conservation property locally on each polytopal element. Section \ref{Section:EE} derives an error equation and    some error estimates for the numerical approximation are established in Section \ref{Section:L2Error}. In Section \ref{Section:Numerics}, various numerical experiments are presented to demonstrate the performance of the $L^p$-primal-dual weak Galerkin method.

This paper follows the standard definitions and notations for Sobolev spaces and norms. Let $D\subset \mathbb{R}^d$ be any open bounded domain with Lipschitz continuous boundary.  Denote by $\|\cdot\|_{L^p(D)}$ the norm in the Sobolev space $W^{0,p}(D)$. Moreover, we use $C$ to denote the generic positive constant independent of meshsize or functions appearing in the equalities.

\section{Weak Formulation and Discrete Weak Gradient}\label{Section:DWG}
In this section, we shall introduce the weak formulation  for the linear convection equation in divergence form (\ref{model}). In addition, we shall briefly review the weak gradient operator as well as its discrete version   \cite{wy}.

The weak formulation for the model problem (\ref{model}) is as follows: Find $u \in L^q(\Omega)$ such that \begin{equation}\label{weakform}
 (u, \bbeta \cdot \nabla\sigma- c \sigma)=\langle g, \bbeta \cdot \bn \sigma\rangle_{\Gamma_-}-(f,\sigma), \qquad \forall \sigma\in W_{0,\Gamma_+}^{1,p}(\Omega).
\end{equation}
 Here  $\Gamma_+=\partial \Omega \setminus \Gamma_-$ is the outflow boundary satisfying $\bbeta \cdot \bn \geq 0$, and $W_{0,\Gamma_+}^{1,p}(\Omega)$ is the subspace of $W^{1,p}(\Omega)$ with vanishing boundary value on $\Gamma_+$; i.e.,
$$
W_{0,\Gamma_+}^{1,p}(\Omega)=\{v \in W^{1,p}(\Omega): v=0 \ \text{on}\ \Gamma_+\}.
 $$

Let $T$ be a polygonal or polyhedral domain with boundary $\partial T$. Denote by $v=\{v_0,v_b\}$ a weak function on $T$ such that $v_0\in L^p(T)$ and $v_b\in L^{p}(\partial T)$. The components $v_0$ and $v_b$ represent the values of $v$ in the interior and on the boundary of $T$, respectively. Note that $v_b$ may not necessarily be the trace of $v_0$ on $\partial T$. Denote by $\W(T)$ the space of weak functions on $T$; i.e.,
\begin{equation*}\label{2.1}
\W(T)=\{v=\{v_0, v_b\}: v_0 \in L^p(T), v_b \in L^{p}(\partial T)\}.
\end{equation*}

The weak gradient of $v\in \W(T)$, denoted by $\nabla_w v$, is defined as a continuous linear functional in the Sobolev space $[W^{1,q}(T)]^d$ satisfying
\begin{equation*}
(\nabla_w  v,\boldsymbol{\psi})_T=-(v_0,\nabla \cdot \boldsymbol{\psi})_T+\langle v_b,\boldsymbol{\psi}\cdot \textbf{n}\rangle_{\partial T},  \qquad \forall \boldsymbol{\psi}\in [W^{1,q}(T)]^d.
\end{equation*}

Denote by $P_r(T)$ the space of polynomials on $T$ with degree no more than $r$. A discrete version of $\nabla_{w} v$  for $v\in \W(T)$, denoted by $\nabla_{w, r, T}v$, is defined as a unique polynomial vector in $[P_r(T)]^d$ such that
\begin{equation}\label{disgradient}
(\nabla_{w, r, T} v, \boldsymbol{\psi})_T=-(v_0, \nabla \cdot \boldsymbol{\psi})_T+\langle v_b, \boldsymbol{\psi} \cdot \textbf{n}\rangle_{\partial T}, \quad\forall\boldsymbol{\psi}\in [P_r(T)]^d,
\end{equation}
which, using the usual integration by parts, yields
\begin{equation}\label{disgradient*}
(\nabla_{w, r, T} v, \boldsymbol{\psi})_T= (\nabla v_0, \boldsymbol{\psi})_T-\langle v_0- v_b, \boldsymbol{\psi} \cdot \textbf{n}\rangle_{\partial T}, \quad\forall\boldsymbol{\psi}\in [P_r(T)]^d,
\end{equation}
provided that $v_0\in W^{1,p}(T)$.

\section{ $L^p$ Primal-Dual Weak Galerkin Algorithm}\label{Section:WGFEM}
Denote by ${\cal T}_h$ a partition of the domain $\Omega \subset {\mathbb R}^d (d=2, 3)$ into polygons in 2D or polyhedra in 3D which is shape regular described as in \cite{wy3655}. Let ${\mathcal E}_h$ be the set of all edges or flat faces in ${\cal T}_h$ and ${\mathcal E}_h^0={\mathcal E}_h \setminus \partial\Omega$ be the set of all interior edges or flat faces. Denote by $h_T$ the meshsize of $T\in {\cal T}_h$ and
$h=\max_{T\in {\cal T}_h}h_T$ the meshsize of the partition ${\cal T}_h$.

For any integer $j\geq 0$, let $W_j(T)$ be the local space of discrete weak functions; i.e.,
$$
W_j(T)=\{\{\sigma_0,\sigma_b\}:\sigma_0\in P_j(T),\sigma_b\in P_j(e), e\subset \partial T\}.
$$
Patching $W_j(T)$ over all the elements $T\in {\cal T}_h$ through a common value $v_b$ on the interior interface $\E_h^0$ gives rise to a global weak finite element space $W_{j,h}$. Let $W_{j,h}^{0, \Gamma_+}$ be the subspace of $W_{j,h}$ with vanishing boundary value on $\Gamma_+$; i.e.,
$$
W_{j,h}^{0, \Gamma_+}=\{\{\sigma_0, \sigma_b\}\in W_{j,h}: \sigma_b|_{e}=0, e\subset \Gamma_+\}.
$$
For any integer $k\ge 1$, let $M_{k-1,h}$ be the space of piecewise polynomials of degree $k-1$; i.e.,
$$
M_{k-1,h}=\{w: w|_T\in P_{k-1}(T), \forall T\in {\cal T}_h\}.
$$

The discrete weak gradient $\nabla_{w,r,T}$ is taken in the polynomial vector space  $[P_r(T)]^d$  with $r=k-1$. For simplicity of notation and without confusion, denote by $\nabla_{w}\sigma$ the discrete weak gradient
$\nabla_{w, k-1, T}\sigma$ for any $\sigma\in W_{j,h}$ computed by
(\ref{disgradient}) on each element $T$; i.e.,
$$
(\nabla_{w}\sigma)|_T= \nabla_{w, k-1, T}(\sigma|_T), \qquad \sigma\in W_{j, h}.
$$

For any $\lambda, \sigma\in W_{j,h}$ and $v\in M_{k-1,h}$, we introduce the
following bilinear forms
\begin{eqnarray}
s(\lambda, \sigma)&=&\sum_{T\in {\cal T}_h}s_T(\lambda, \sigma),\label{stabilizer}
\\
b(v, \sigma)&=&\sum_{T\in {\cal T}_h}b_T(v, \sigma),\label{b-form}
\end{eqnarray}
where
\begin{equation}\label{stabilizer-local}
\begin{split}
s_T(\lambda, \sigma) =& \rho h_T^{1-p}\int_{\partial T} |\lambda_0-\lambda_b |^{p-1}sgn(\lambda_0-\lambda_b)(\sigma_0-\sigma_b)ds\\
&  + \tau \int_{ T} |\bbeta\cdot\nabla\lambda_0-c\lambda_0|^{p-1}sgn(\bbeta\cdot\nabla\lambda_0
-c\lambda_0)(\bbeta\cdot\nabla\sigma_0-c\sigma_0)dT,
\end{split}
\end{equation}
\begin{equation}\label{b-form-local}
\begin{split}
b_T(v, \sigma) =(v, \bbeta \cdot \nabla_w \sigma -c \sigma_0)_T,
\end{split}
\end{equation}
 where   $\rho>0$ and $\tau\geq 0$.

The $L^p$-PDWG numerical scheme for the linear convection equation (\ref{model}) based on the variational formulation (\ref{weakform}) is as follows:
\begin{algorithm}
Find $(u_h;\lambda_h)\in M_{k-1,h} \times W_{j,h}^{0, \Gamma_+}$ such that
\begin{eqnarray}\label{32}
s(\lambda_h, \sigma)+b(u_h,\sigma)&=& \sum_{e\subset \Gamma_-}\langle \sigma_b, \bbeta \cdot \bn g  \rangle_{e}-(f,\sigma_0), \qquad \forall\sigma\in W_{j,h}^{0, \Gamma_+},\\
b(v,\lambda_h)&=&0,\qquad \qquad\qquad\qquad \qquad \qquad \forall v\in M_{k-1,h}.\label{2}
\end{eqnarray}
\end{algorithm}

\section{Solution Existence and Uniqueness}\label{Section:EU}

The adjoint problem for the linear transport equation is as follows:  Find $\Psi$ such that
\begin{eqnarray}\label{mo}
\bbeta \cdot \nabla \Psi-c\Psi&=&\theta, \qquad \mbox{ in } \Omega,\\
\Psi & = &0, \qquad\mbox { on } \Gamma_+, \label{mo-2}
\end{eqnarray}
where $\theta \in L^q(\Omega)$ is a given function. We assume that the adjoint problem \eqref{mo}-\eqref{mo-2} has a  unique solution.

 On each element $T$, denote by $Q_0$ the $L^2$ projection operator onto $P_j(T)$. For each edge or face $e\subset\partial T$, denote by $Q_b$ the $L^2$ projection operator onto $P_{j}(e)$. For any $w\in H^1(\Omega)$,  denote by $Q_h w$  the $L^2$ projection of $w$ onto the finite element space $W_{j,h}$ such that on each element $T$, $Q_hw=\{Q_0w,Q_bw\}$.
 Denote by ${\cal Q}_h$ the $L^2$ projection operator onto the space $M_{k-1,h}$.

\begin{lemma}\label{Lemma5.1} \cite{wy3655} For $j\ge k-1$, the $L^2$ projection operators $Q_h$ and ${\cal Q}_h$ satisfy the following commutative property:
\begin{equation}\label{l}
\nabla_{w}(Q_h w) = {\cal Q}_h(\nabla w), \qquad \forall w\in H^1(T).
\end{equation}
\end{lemma}

For the convenience of analysis, in what follows of this paper, we assume that the convection vector $\bbeta$ and the reaction coefficient $c$ are   piecewise constants with respect to the partition $\T_h$.
However, the analysis and results can be generalized  to piecewise smooth cases for the convection vector $\bbeta$ and the reaction coefficient $c$ without any difficulty.

 In order to prove the existence for the numerical solution $(u_h;\lambda_h)$ arising from $L^p$-PDWG scheme (\ref{32})-(\ref{2}), we introduce a functional defined in the finite element spaces $M_{k-1,h}\times W_{j,h}^{0,\Gamma_+}$; i.e.,
$$J(v,\sigma)=\frac{1}{p}s(\sigma,\sigma)+b(v,\sigma)-(F,\sigma),$$
where $(F,\sigma)=\sum_{e\subset\Gamma_-}\langle\sigma_b,\bbeta\cdot\bn g\rangle_{e}-(f,\sigma_0).$

Using \eqref{2} gives
\begin{equation*}
J(v,\lambda_h)=\frac{1}{p}s(\lambda_h,\lambda_h)-(F,\lambda_h),\quad\forall v\in M_{k-1,h},
\end{equation*}
which implies
\begin{equation}\label{15:31}
J(v,\lambda_h)=J(u_h,\lambda_h).
\end{equation}

 It follows from (\ref{32}) that $D_\sigma J(u_h,\lambda_h)(\sigma)=0$ holds true for any $\sigma\in W_{j,h}^{0, \Gamma_+}$, where $D_\sigma J(u_h,\lambda_h)(\sigma)$ represents the Gateaux partial derivative at $\lambda_h$ in the direction of $\sigma$. From the convexity of $J(u_h,\sigma)$ in the direction of $\sigma$, we have
\begin{equation}\label{15:33}
J(u_h,\lambda_h)\leq J(u_h,\sigma).
\end{equation}
Combining \eqref{15:31} with \eqref{15:33} gives
$$
J(v,\lambda_h)\leq J(u_h,\lambda_h)\leq J(u_h,\sigma),
$$
which implies that $(u_h;\lambda_h)$ is the saddle point of $J(v,\sigma)$.

Thus, the numerical scheme (\ref{32})-(\ref{2}) can be formulated as a $\min$-$\max$ problem that seeks $(u_h;\lambda_h)\in M_{k-1,h}\times W_{j,h}^{0,\Gamma_+}$ such that
\begin{equation}\label{15:35}
(u_h;\lambda_h)=\arg  \min_{\sigma\in W_{j,h}^{0,\Gamma_+}}   \max_{v\in M_{k-1,h}}J(v,\sigma).
\end{equation}
It is easy to conclude that the $\min$-$\max$ problem \eqref{15:35} has a numerical solution due to the convexity of \eqref{15:35} so that $(u_h;\lambda_h)$ satisfies (\ref{32})-(\ref{2}).

\begin{theorem}\label{thmunique1}
 Assume  that the linear transport problem \eqref{model} has a unique solution. The following results hold true:
\begin{itemize}
\item If $j=k-1$, then the $L^p$-PDWG algorithm (\ref{32})-(\ref{2}) has one and only one solution for $\tau\geq0$.
\item If $j= k$, then the $L^p$-PDWG algorithm (\ref{32})-(\ref{2}) has a unique solution for $\tau>0$.
    \end{itemize}
\end{theorem}

\begin{proof}
Let $(u^{(1)}_h;\lambda^{(1)}_h)$ and $(u^{(2)}_h;\lambda^{(2)}_h)$ be two different solutions of (\ref{32})-(\ref{2}).  Denote
\[
    \epsilon_h=\lambda^{(1)}_h-\lambda^{(2)}_h=\{\epsilon_0,\epsilon_b\},\ \ e_h=u^{(1)}_h-u^{(2)}_h.
\]
    For any constants $\theta_1,\theta_2$,
  we choose $\sigma=\theta_1\lambda^{(1)}_h+\theta_2\lambda^{(2)}_h$ in \eqref{32} and use \eqref{2} to obtain
\[
     s(\lambda^{(1)}_h, \theta_1\lambda^{(1)}_h+\theta_2\lambda^{(2)}_h)-s(\lambda^{(2)}_h, \theta_1\lambda^{(1)}_h+\theta_2\lambda^{(2)}_h)=0.
\]
In particular,  by taking $(\theta_1,\theta_2)=(1,0), (0,1)$, we have
  \begin{equation}\label{33}
     s(\lambda^{(1)}_h, \lambda^{(1)}_h)=s(\lambda^{(2)}_h,\lambda^{(1)}_h),\ \  s(\lambda^{(2)}_h, \lambda^{(2)}_h)=s(\lambda^{(1)}_h,\lambda^{(2)}_h),
\end{equation}
   which yields, together with Young's inequality  $|AB|\le \frac{|A|^p}{p}+\frac{|B|^q}{q}$  with $\frac{1}{p}+\frac{1}{q}=1$, that  \[
   s(\lambda^{(1)}_h, \lambda^{(1)}_h)\le \frac{s(\lambda^{(2)}_h,\lambda^{(2)}_h)}{q}+\frac{s(\lambda^{(1)}_h,\lambda^{(1)}_h)}{p},\ \
    s(\lambda^{(2)}_h, \lambda^{(2)}_h)\le \frac{ s(\lambda^{(1)}_h,\lambda^{(1)}_h)}{q}+\frac{ s(\lambda^{(2)}_h,\lambda^{(2)}_h)}{p},
\]
which yields
\begin{equation}\label{35}
    s(\lambda^{(1)}_h, \lambda^{(1)}_h)=s(\lambda^{(2)}_h, \lambda^{(2)}_h).
\end{equation}
On the other hand,  for any two real numbers $A,B$,  there holds
  \[
   \left|\frac{A+B}{2}\right|^p\le (|A|^p+|B|^p)/2,
 \]
and the equality holds true if and only if $A=B$.  It follows that
 \begin{equation}\label{34}
     s(\frac{\lambda^{(1)}_h+\lambda^{(2)}_h}{2}, \frac{\lambda^{(1)}_h+\lambda^{(2)}_h}{2}) \le  \frac{1}{2} \big( s( \lambda^{(1)}_h, \lambda^{(1)}_h)+ s( \lambda^{(2)}_h, \lambda^{(2)}_h)\big)=s( \lambda^{(1)}_h, \lambda^{(1)}_h).
 \end{equation}
      By \eqref{33}-\eqref{35} and Young's inequality,
\begin{eqnarray*}
    s( \lambda^{(1)}_h, \lambda^{(1)}_h)
     &= &\frac{1}{2}\big ( s( \lambda^{(1)}_h, \lambda^{(1)}_h)+ s( \lambda^{(1)}_h, \lambda^{(2)}_h)\big)
=s( \lambda^{(1)}_h, \frac{ \lambda^{(1)}_h+\lambda^{(2)}_h}{2}) \\
  & \le & \frac 1q{s(\lambda^{(1)}_h,\lambda^{(1)}_h)} +\frac 1p {s(\frac{\lambda^{(1)}_h+\lambda^{(2)}_h}{2}, \frac{\lambda^{(1)}_h+\lambda^{(2)}_h}{2})},
\end{eqnarray*}
  which indicates that
 \[
     s( \lambda^{(1)}_h, \lambda^{(1)}_h)\le {s(\frac{\lambda^{(1)}_h+\lambda^{(2)}_h}{2}, \frac{\lambda^{(1)}_h+\lambda^{(2)}_h}{2})}.
 \]
 From \eqref{35}-\eqref{34}, we easily obtain
\begin{eqnarray*}
    s(\frac{\lambda^{(1)}_h+\lambda^{(2)}_h}{2}, \frac{\lambda^{(1)}_h+\lambda^{(2)}_h}{2}) =s( \lambda^{(1)}_h, \lambda^{(1)}_h)=s( \lambda^{(2)}_h, \lambda^{(2)}_h).
 \end{eqnarray*}
    The above equality holds true if and only if
\begin{equation}\label{e_0-1}
\begin{split}
       \epsilon_0 =& \epsilon_b, \mbox{ on } \pT,\\
       \bbeta \cdot\nabla  \epsilon_0 -c \epsilon_0  =&0, \mbox{ in } T,
\end{split}
\end{equation}
for the case of $\tau>0$,
and
\begin{eqnarray}\label{e_0-2}
       \epsilon_0  =  \epsilon_b, \mbox{ on } \pT,
 \end{eqnarray}
for the case of $\tau=0.$

{\em Case 1:} $j=k-1$ and $\tau=0.$ It follows from (\ref{2}) and (\ref{disgradient*}) that
\begin{equation}\label{new:001-0}
\begin{split}
0=& b(v, \epsilon_h)\\
 =& \sum_{T\in {\cal T}_h} (v, \bbeta \cdot \nabla_w \epsilon_h-c \epsilon_0)_T\\
= & \sum_{T\in {\cal T}_h}  ( \nabla  \epsilon_0, \bbeta v)_T-\langle  \epsilon_0- \epsilon_b, \bbeta v\cdot \bn\rangle_{\partial T}-(v, c \epsilon_0)_T\\
 = &  \sum_{T\in {\cal T}_h}  (\bbeta \cdot \nabla  \epsilon_0 -c{\cal Q}_h\epsilon_0,   v)_T,
\end{split}
\end{equation}
where we have used \eqref{e_0-2} on each  $\partial T$. By taking $v=\bbeta \cdot \nabla  \epsilon_0 -c{\cal Q}_h\epsilon_0$ we obtain
$$
\bbeta \cdot \nabla  \epsilon_0 -c{\cal Q}_h\epsilon_0=0
$$
on each element $T\in {\cal T}_h$. From $ \epsilon_0= \epsilon_b$ on each  $\partial T$, we have $ \epsilon_0\in H^1(\Omega)$ so that
\begin{equation}\label{new:001}
\bbeta \cdot \nabla  \epsilon_0-c \epsilon_0=0, \quad \text{in }  \Omega,
\end{equation}
where we used ${\cal Q}_h\epsilon_0 =\epsilon_0$ for the case of $j=k-1$.    It follows from $ \epsilon_b|_{\Gamma_+} = 0$ and $ \epsilon_0= \epsilon_b$ on each  $\partial T$ that $ \epsilon_0|_{\Gamma_+} = 0$, which, using the solution uniqueness of the adjoint problem \eqref{mo}-\eqref{mo-2}, gives $ \epsilon_0 \equiv 0$ in $\Omega$.

{\em Case 2:} $j = k$ and $\tau>0.$ It follows from \eqref{e_0-1} that
\begin{equation*}\label{aa-001}
\bbeta \cdot \nabla  \epsilon_0 -c \epsilon_0=0\quad  \mbox{strongly in $\Omega.$}
\end{equation*}
Using the solution uniqueness of the adjoint problem \eqref{mo}-\eqref{mo-2} and $ \epsilon_0|_{\Gamma_+} = 0$,  we have $ \epsilon_0 \equiv 0$ in $\Omega$.

Next, we show $e_h=0$. To this end, using  $\lambda_h^{(1)}=\lambda_h^{(2)}$ and \eqref{32}, we have
 $$
b(e_h,\sigma)=s(\lambda^{(1)}_h,\sigma)- s(\lambda^{(2)}_h,\sigma) +b(e_h,\sigma)=0, \qquad \forall\sigma\in W_{j,h}^{0, \Gamma_+}.
$$
From (\ref{disgradient}) we obtain
\begin{equation}\label{eee}
\begin{split}
0=& b(e_h,\sigma)\\
 =& \sum_{T\in {\cal T}_h} (e_h, \bbeta \cdot \nabla_w\sigma-c\sigma_0)_T\\
= & \sum_{T\in {\cal T}_h} -(\sigma_0,  \nabla\cdot(\bbeta e_h))_T+\langle \sigma_b, \bbeta e_h\cdot \bn\rangle_{\partial T}-(e_h, c\sigma_0)_T\\
= & -\sum_{T\in {\cal T}_h}(\sigma_0, \nabla\cdot(\bbeta e_h)+ce_h)_T+\sum_{e\subset {\cal E}_h \setminus \Gamma_+}\langle \sigma_b, \ljump\bbeta e_h\cdot \bn\rjump\rangle_{e},\\
\end{split}
\end{equation}
where we have used $\sigma_b=0$ on $\Gamma_+$ on the last line, and $\ljump\bbeta e_h\cdot \bn\rjump$ is the jump of $\bbeta e_h\cdot \bn$ on $e\subset {\cal E}_h \setminus \Gamma_+$ in the sense that $\ljump\bbeta e_h\cdot \bn\rjump=\bbeta e_h|_{T_1}\cdot \bn_1 + \bbeta e_h|_{T_2}\cdot \bn_2$ for $e=\partial T_1\cap\partial T_2\subset  {\cal E}_h^0$ with $\bn_1$ and $\bn_2$ being the unit outward normal directions to $\partial T_1$ and $\partial T_2$, respectively, and $\ljump\bbeta e_h\cdot \bn\rjump=\bbeta e_h\cdot \bn$ for $e\subset \Gamma_-$. By setting  $\sigma_0=-(\nabla \cdot (\bbeta  e_h)+ce_h)$  on each $T\in {\cal T}_h$ and  $\sigma_b= \ljump  \bbeta e_h  \cdot \bn\rjump$  on each $e \subset {\cal E}_h \setminus \Gamma_+$, we may rewrite (\ref{eee}) as follows:
\begin{equation*}
\begin{split}
0= \sum_{T\in {\cal T}_h}  \|\nabla \cdot (\bbeta  e_h)+ce_h\|_T^2 + \sum_{e \subset {\cal E}_h \setminus \Gamma_+}  \|\ljump\bbeta e_h\cdot \bn\rjump\|_e^2,
\end{split}
\end{equation*}
which gives $\nabla \cdot (\bbeta  e_h)+ce_h=0$ on each $T\in {\cal T}_h$,  $\ljump\bbeta e_h\cdot \bn\rjump=0$ on each $e \subset {\cal E}_h^0$, and $\bbeta e_h\cdot \bn=0$ on each $e\subset\Gamma_-$. This implies that  $\nabla \cdot (\bbeta  e_h)+ce_h=0$ in $\Omega$ and $e_h=0$ on $\Gamma_-$. Thus, from the solution uniqueness assumption, we have $e_h\equiv 0$ in $\Omega$. This is equivalent to $u_h^{(1)}=u_h^{(2)}$.

This completes the proof of the theorem.
\end{proof}

\section{Mass Conservation} \label{Section:MC}

The linear convection equation (\ref{model}) can be rewritten in a conservative form as follows:
\begin{eqnarray}\label{eq1}
\nabla \cdot \textbf{F} +cu&=&f, \\
\label{eq2}
 \textbf{F}&=&\bbeta u.
\end{eqnarray}
On each element $T\in \T_h$, we may integrate (\ref{eq1}) over $T$ to obtain the integral form of the mass conservation:
\begin{equation}\label{mas}
\int_{\partial T}\textbf{F} \cdot \bn ds+\int_T cu dT=\int_T fdT.
\end{equation}

We claim that  the numerical solution arising from the $L^p$-primal-dual weak Galerkin scheme (\ref{32})-(\ref{2}) for the linear convection problem (\ref{model}) retains the local mass conservation property (\ref{mas}) with a numerical flux $\bF_h$. To this end, for any given $T\in {\cal T}_h$, by choosing a test function  $\sigma=\{\sigma_0, \sigma_b=0\}$  in (\ref{32})  such that $\sigma_0=1$ on $T$ and $\sigma_0=0$ elsewhere, we obtain
\begin{equation*}
    \begin{split}
    &h_T^{1-p} \langle |\lambda_0-\lambda_b|^{p-1}sgn(\lambda_0-\lambda_b), 1-0  \rangle_{\partial T}\\&-\tau(|\bbeta\cdot\nabla\lambda_0-c\lambda_0|^{p-1}sgn(\bbeta\cdot\nabla\lambda_0-c\lambda_0),c)_T + (u_h, \bbeta \cdot \nabla_w \sigma-c\cdot 1 )_T=-(f, 1)_T.
\end{split}
\end{equation*}
It follows from (\ref{disgradient}) and the usual integration by parts  that
\begin{equation}\label{e1}
\begin{split}
&(f, 1)_T\\
= &-h_T^{1-p} \langle |\lambda_0-\lambda_b|^{p-1}sgn(\lambda_0-\lambda_b), 1  \rangle_{\partial T} + (\nabla \cdot(\bbeta u_h), 1)_T + (cu_h, 1)_T\\
&+\tau(|\bbeta\cdot\nabla\lambda_0-c\lambda_0|^{p-1}sgn(\bbeta\cdot\nabla\lambda_0-c\lambda_0),c)_T\\
=&-h_T^{1-p} \langle |\lambda_0-\lambda_b|^{p-1}sgn(\lambda_0-\lambda_b), 1 \rangle_{\partial T} + \langle \bbeta u_h \cdot \bn, 1\rangle_{\partial T}\\
&+(c(u_h+\tau |\bbeta\cdot\nabla\lambda_0-c\lambda_0|^{p-1}sgn(\bbeta\cdot\nabla\lambda_0-c\lambda_0)), 1)_T\\
=& \langle (-h_T^{1-p} |\lambda_0-\lambda_b|^{p-1}sgn(\lambda_0-\lambda_b)\bn +\bbeta u_h)\cdot \bn, 1  \rangle_{\partial T} \\
&+ (c(u_h+\tau |\bbeta\cdot\nabla\lambda_0-c\lambda_0|^{p-1}sgn(\bbeta\cdot\nabla\lambda_0-c\lambda_0)), 1)_T,
\end{split}
\end{equation}
where $\bn$ is the outward normal direction to $\partial T$. The equation (\ref{e1}) implies that the $L^p$-primal-dual weak Galerkin algorithm (\ref{32})-(\ref{2}) conserves mass with a numerical solution and a numerical flux given by
$$
\tilde u_h = u_h+\tau |\bbeta\cdot\nabla\lambda_0-c\lambda_0|^{p-1}sgn(\bbeta\cdot\nabla\lambda_0-c\lambda_0),$$
$$
\textbf{F}_h|_\pT=- h_T^{1-p} |\lambda_0-\lambda_b|^{p-1}sgn(\lambda_0-\lambda_b)\bn +\bbeta u_h.
$$

It remains to show that the numerical flux $\bF_h \cdot \bn$ is continuous across each interior edge or flat face. To this end, we choose a test function $\sigma=\{\sigma_0=0, \sigma_b\}$ in (\ref{32}) such that
$\sigma_b$ is arbitrary on one interior edge or flat face $e=\partial T_1\cap \partial T_2$, and $\sigma_b=0$ elsewhere, to obtain
\begin{equation*}\label{mass}
\begin{split}
0=&h_{T_1}^{1-p} \langle |\lambda_0-\lambda_b|^{p-1}sgn(\lambda_0-\lambda_b), -\sigma_b\rangle_{e\cap\partial T_1}+h_{T_2}^{1-p} \langle |\lambda_0-\lambda_b|^{p-1}sgn(\lambda_0-\lambda_b), -\sigma_b\rangle_{e\cap\partial T_2}\\
&+(u_h, \bbeta \cdot \nabla_w \sigma)_{T_1\cup T_2} \\
=&h_{T_1}^{1-p} \langle |\lambda_0-\lambda_b|^{p-1}sgn(\lambda_0-\lambda_b), -\sigma_b\rangle_{e\cap\partial T_1}+h_{T_2}^{1-p} \langle  |\lambda_0-\lambda_b| ^{p-1}sgn(\lambda_0-\lambda_b), -\sigma_b\rangle_{e\cap\partial T_2}\\
&+\langle \bbeta u_h \cdot \bn_{T_1}, \sigma_b\rangle_{e\cap\partial T_1}+\langle \bbeta u_h \cdot \bn_{T_2}, \sigma_b\rangle_{e\cap\partial T_2}\\
 =&\langle (\bbeta u_h-h_{T_1}^{1-p} |\lambda_0-\lambda_b|^{p-1}sgn(\lambda_0-\lambda_b)\bn_{T_1}  ) \cdot \bn_{T_1},  \sigma_b\rangle_{e\cap\partial T_1}\\
  &+ \langle (\bbeta u_h -h_{T_2}^{1-p}|\lambda_0-\lambda_b|^{p-1}sgn(\lambda_0-\lambda_b)\bn_{T_2}  )\cdot \bn_{T_2},  \sigma_b\rangle_{e\cap\partial T_2}\\
=&  \langle \bF_h|_{\pT_1} \cdot \bn_{T_1},  \sigma_b\rangle_{e\cap\partial T_1}+ \langle\bF_h|_{\pT_2} \cdot \bn_{T_2}, \sigma_b\rangle_{e\cap\partial T_2},
\end{split}
\end{equation*}
where we have used (\ref{disgradient}), $\bn_{T_1}$ and $\bn_{T_2}$ are the unit outward normal directions along $e=\partial T_1\cap \partial T_2$ pointing exterior to $T_1$ and $T_2$, respectively.  This shows that
$$
\bF_h|_{\pT_1} \cdot \bn_{T_1} + \bF_h|_{\pT_2} \cdot \bn_{T_2} = 0\qquad \mbox{on } e=\pT_1\cap\pT_2,
$$
 and hence the continuity of the numerical flux along the normal direction on each interior edge or flat face.

The result can be summarized as follows.

\begin{theorem}\label{THM:conservation} Let $(u_h;\lambda_h)$ be the numerical solution of the linear convection model problem \eqref{model} arising from    $L^p$-PDWG finite element scheme (\ref{32})-(\ref{2}). Define a new numerical approximation and a numerical flux function as follows:
\begin{eqnarray*}
\tilde u_h &:=& u_h+\tau |\bbeta\cdot\nabla\lambda_0-c\lambda_0|^{p-1}sgn(\bbeta\cdot\nabla\lambda_0-c\lambda_0),\qquad \mbox{in} \ T,  \ T\in \T_h,\\
\textbf{F}_h|_\pT&:=& -h_T^{1-p} |\lambda_0-\lambda_b|^{p-1}sgn(\lambda_0-\lambda_b)\bn +\bbeta u_h, \quad\mbox{on } \pT, \ T\in \T_h.
\end{eqnarray*}
Then, the flux approximation $\textbf{F}_h$ is continuous across each interior edge or flat face in the normal direction, and the following conservation property is satisfied:
\begin{equation}\label{mas-discrete}
\int_{\partial T}\textbf{F}_h \cdot \bn ds+\int_T c\tilde u_h dT=\int_T fdT.
\end{equation}
\end{theorem}

\section{Error Equations}\label{Section:EE}
Let $u$ and $(u_h; \lambda_h) \in M_{k-1,h}\times W_{j, h}^{0, \Gamma_+}$ be the exact solution of (\ref{model}) and the numerical solution arising from the $L^p$-primal-dual weak Galerkin scheme (\ref{32})-(\ref{2}), respectively. Note that the exact solution  of the Lagrangian multiplier is  $\lambda=0$. The error functions for the primal variable $u$ and the dual variable $\lambda$ are thus given by
\begin{align*}
e_h&=u_h-{\cal Q} _hu,
\\
\epsilon_h&=\lambda_h-Q_h\lambda=\lambda_h.
\end{align*}

\begin{lemma}\label{errorequa}
Let $u$ and $(u_h; \lambda_h) \in M_{k-1,h}\times W_{j, h}^{0, \Gamma_+}$ be the exact solution of (\ref{model}) and the numerical solution arising from the $L^p$-primal-dual weak Galerkin scheme (\ref{32})-(\ref{2}), respectively. Then, the error functions $e_h$ and $\epsilon_h$ satisfy the following equations:
\begin{eqnarray}\label{sehv}
s( \epsilon_h , \sigma)+b(e_h, \sigma )&=&\ell_u(\sigma)
,\qquad \forall\sigma\in W^{0, \Gamma_+}_{j,h},\\
b(v, \epsilon_h )&=&0,\qquad\qquad \forall v\in M_{k-1,h}, \label{sehv2}
\end{eqnarray}
where
\begin{equation}\label{lu}
\qquad \ell_u(\sigma) = \sum_{T\in {\cal T}_h}\langle \bbeta (u-{\cal Q}_h u)  \cdot \bn ,\sigma_b-\sigma_0 \rangle_{\partial T} +   ({\cal Q}_hu-u, c\sigma_0 )_T.
\end{equation}
\end{lemma}
\begin{proof}
From (\ref{2}) we have
\begin{align*}
b(v, \epsilon_h) = 0,\qquad\forall v\in M_{k-1,h},
\end{align*}
which gives rise to the equation (\ref{sehv2}).

Next, by subtracting $b( {\cal Q}_hu, \sigma)$ from both sides of (\ref{32}) we arrive at
\begin{equation*}
\begin{split}
& s( \lambda_h-Q_h\lambda, \sigma)+b(u_h-{\cal Q}_hu, \sigma)  \\
=& -(f,\sigma_0)-b(  {\cal Q}_hu, \sigma)+\sum_{e\subset \Gamma_-}\langle \sigma_b, \bbeta \cdot\bn g \rangle_{e}\\
 =&  -(f,\sigma_0) -\sum_{T\in {\cal T}_h}  ({\cal Q}_hu, \bbeta \cdot \nabla_w \sigma-c\sigma_0)_T+\sum_{e\subset \Gamma_-}\langle \sigma_b, \bbeta \cdot\bn g \rangle_{e}\\
                     =&  -(f,\sigma_0)+\sum_{T\in {\cal T}_h}- (  \bbeta{\cal Q}_hu, \nabla \sigma_0)_T+\langle \bbeta {\cal Q}_hu \cdot \bn,  \sigma_0 -\sigma_b\rangle_{\partial T} \\
                     &+ ({\cal Q}_hu, c\sigma_0)_T +\sum_{e\subset \Gamma_-}\langle \sigma_b, \bbeta \cdot\bn g \rangle_{e}\\
                                      =&  -(f,\sigma_0)+ \sum_{T\in {\cal T}_h}-(   \bbeta u, \nabla \sigma_0)_T+\langle \bbeta {\cal Q}_hu \cdot \bn,  \sigma_0 -\sigma_b\rangle_{\partial T} \\
                                      &+ ( {\cal Q}_hu, c\sigma_0)_T +\sum_{e\subset \Gamma_-}\langle \sigma_b, \bbeta \cdot\bn g \rangle_{e}\\
                                      =&  -(f,\sigma_0)+ \sum_{T\in {\cal T}_h}(  \nabla  \cdot ( \bbeta u)+c{\cal Q}_hu, \sigma_0)_T- \langle   \bbeta u  \cdot \bn,  \sigma_0 -\sigma_b\rangle_{\partial T}\\& +\langle \bbeta {\cal Q}_hu \cdot \bn,  \sigma_0 -\sigma_b\rangle_{\partial T}
                                      +\sum_{e\subset \Gamma_-} \{-\langle   \bbeta u  \cdot \bn,   \sigma_b\rangle_{e} +\langle \sigma_b, \bbeta \cdot\bn g \rangle_{e} \}\\
                                      =&\sum_{T\in {\cal T}_h} \langle  \bbeta(u- {\cal Q}_h u) \cdot \bn,  \sigma_b -\sigma_0\rangle_{\partial T} +  ({\cal Q}_hu-u, c\sigma_0 )_T,
  \end{split}
\end{equation*}
where we have used (\ref{disgradient*}), the usual integration by parts, and the facts that $ \nabla  \cdot (\bbeta u)+cu=f$  in $\Omega$, $u=g$ on $\Gamma_-$, and $\sigma_b=0$ on $\Gamma_+$. This completes the proof of the lemma.
\end{proof}

\section{ Error Estimates} \label{Section:L2Error}
Recall that $\T_h$ is a shape-regular finite element partition of
the domain $\Omega$. For any $T\in\T_h$ and $\nabla w\in L^{q}(T)$ with $q>1$, the following trace inequality holds true:
\begin{equation}\label{trace-inequality}
\|w\|^q_{L^q(\pT)}\leq C
h_T^{-1}(\|w\|_{L^{q}(T)}^q+h_T^{q} \| \nabla w\|_{L^{q}(T)}^q).
\end{equation}
If $w$ is a polynomial, using inverse inequality, we have
\begin{equation}\label{trace}
\|w\|^q_{L^q(\pT)}\leq C
h_T^{-1}\|w\|_{L^{q}(T)}^q.
\end{equation}

 For any $\sigma\in W_{j,h}$ and $v\in M_{k-1,h}$, we introduce two seminorms given by
\begin{eqnarray*}
\3bar\sigma\3bar&=&s(\sigma,\sigma)^{1/p},\label{semi-normm}\\
\|v\|_{M_h}&=&\Big(\sum_{T\in {\cal T}_h}h_T^{q}\int_{T}|\nabla\cdot(\bbeta v)+cv|^{q-1}sgn(\nabla\cdot(\bbeta v)+cv)(\nabla\cdot(\bbeta v)+cv)dT\\
&+&\sum_{e\subset {\cal E}_h \setminus \Gamma_+}h_T\int_{e}|\ljump\bbeta v\cdot\bn\rjump|^{q-1}sgn(\ljump\bbeta v\cdot\bn\rjump)\ljump\bbeta v\cdot\bn\rjump ds\Big)^{1/q}.\label{semi-normm-3}
\end{eqnarray*}

\begin{lemma}
 For any $v\in M_{k-1,h}$, the seminorm $\|v\|_{M_h}$ defines a norm.
\end{lemma}
\begin{proof}
    The proof is similar to Lemma 7.1 in \cite{ww2020}.
\end{proof}
\begin{lemma}\label{inf-sup-0}
(inf-sup condition) Let $j=k-1$ or $j=k$  with $k\geq 1$. Then, for any $v\in M_{k-1,h}$, there exists a weak function $\tilde{\sigma}\in W_{j,h}^{0, \Gamma_+}$ satisfying
\begin{eqnarray}
b(v,\tilde{\sigma})&=&\|v\|_{M_h}^q,\label{inf-sup-1}\\
\3bar\tilde{\sigma}\3bar^p&\leq& C\|v\|_{M_h}^q\label{inf-sup-2}.
\end{eqnarray}
  \end{lemma}

\begin{proof}
It follows from any $\sigma\in W_{j,h}^{0, \Gamma_+}$ and the definition of weak gradient operator \eqref{disgradient} that
\begin{equation*}
\begin{split}
b(v,\sigma)=&\sum_{T\in {\cal T}_h}(v,\bbeta\cdot\nabla_w\sigma-c\sigma_0)_T\\
=&\sum_{T\in {\cal T}_h}-(\sigma_0,\nabla\cdot(\bbeta v))_T+\langle \sigma_b,\bbeta v\cdot\bn\rangle_{\partial T}-(c v,\sigma_0)_T\\
=&-\sum_{T\in {\cal T}_h}(\sigma_0,\nabla\cdot(\bbeta v)+cv)_T+
\sum_{e\subset {\cal E}_h \setminus \Gamma_+}\langle\sigma_b,\ljump\bbeta v\cdot \bn\rjump\rangle_e,
\end{split}
\end{equation*}
where the fact $\sigma_b=0$ on $\Gamma_+$ is also used in the last step.

 By taking $\tilde{\sigma}=\{\tilde{\sigma}_0,\tilde{\sigma}_b\}$ where $\tilde{\sigma}_0=-h_T^{q}Q_0(\sgn(\nabla\cdot(\bbeta v)+cv)|\nabla\cdot(\bbeta v)+cv|^{q-1})\in P_j(T)$ on each $T$ and $\tilde{\sigma}_b=h_TQ_b(\sgn(\ljump\bbeta v\cdot \bn\rjump)|\ljump\bbeta v\cdot \bn\rjump|^{q-1})\in P_j(e)$ on each $e\subset {\cal E}_h \setminus \Gamma_+$, we have from the definitions of $Q_0$ and $Q_b$ that
\begin{equation*}
\begin{split}
b(v,\tilde{\sigma})
=&\sum_{T\in {\cal T}_h}h_T^{q}\int_TQ_0(\sgn(\nabla\cdot(\bbeta v)+cv)|\nabla\cdot(\bbeta v)+cv|^{q-1})(\nabla\cdot(\bbeta v)+cv)dT\\
&+\sum_{e\subset {\cal E}_h \setminus \Gamma_+}h_T\int_eQ_b(\sgn(\ljump\bbeta v\cdot \bn\rjump)|\ljump\bbeta v\cdot \bn\rjump|^{q-1})\ljump\bbeta v\cdot \bn\rjump ds\\
=&\sum_{T\in {\cal T}_h}h_T^{q}\int_T\sgn(\nabla\cdot(\bbeta v)+cv)|\nabla\cdot(\bbeta v)+cv|^{q-1}(\nabla\cdot(\bbeta v)+cv)dT\\
&+\sum_{e\subset {\cal E}_h \setminus \Gamma_+}h_T\int_e\sgn(\ljump\bbeta v\cdot \bn\rjump)|\ljump\bbeta v\cdot \bn\rjump|^{q-1}\ljump\bbeta v\cdot \bn\rjump ds\\
=&\|v\|_{M_h}^q,
\end{split}
\end{equation*}
which leads to the first equality \eqref{inf-sup-1}.

Note that for two real numbers, one has
\begin{equation}\label{ab}
 |a+b|^p\leq 2^{p-1}(|a|^p+|b|^p).
   \end{equation}
 To verify \eqref{inf-sup-2},  using  \eqref{ab}, trace inequality \eqref{trace}, $p+q=pq$, and inverse inequality obtains

\begin{equation*}
\begin{split}
\3bar\tilde{\sigma}\3bar^p=&\sum_{T\in {\cal T}_h}\rho h_T^{1-p}\int_{\partial T} |\tilde{\sigma}_0-\tilde{\sigma}_b |^{p-1}sgn(\tilde{\sigma}_0-\tilde{\sigma}_b)(\tilde{\sigma}_0-\tilde{\sigma}_b)ds\\
&+\tau\int_{T}|\bbeta\cdot\nabla\tilde{\sigma}_0-c\tilde{\sigma}_0|^{p-1}sgn(\bbeta\cdot\nabla\tilde{\sigma}_0-c\tilde{\sigma}_0)
 (\bbeta\cdot\nabla\tilde{\sigma}_0-c\tilde{\sigma}_0)dT\\
\leq &C\sum_{T\in {\cal T}_h}\rho h_T^{1-p}\int_{\partial T} |\tilde{\sigma}_0-\tilde{\sigma}_b |^{p}ds+\tau\int_{T}|\bbeta\cdot\nabla\tilde{\sigma}_0-c\tilde{\sigma}_0|^{p}dT\\
\leq&C\sum_{T\in {\cal T}_h}\rho h_T^{1-p}\int_{\partial T}2^{p-1}(|\tilde{\sigma}_0|^p+|\tilde{\sigma}_b |^{p})ds+\tau\int_{T}2^{p-1}(|\bbeta\cdot\nabla\tilde{\sigma}_0|^p+|c\tilde{\sigma}_0|^{p})dT\\
\leq&C\sum_{T\in {\cal T}_h}h_T^{1-p}\int_{\partial T}|\tilde{\sigma}_0|^pds+h_T^{1-p}\int_{\partial T}|\tilde{\sigma}_b|^pds
+\int_{T}|\nabla\tilde{\sigma}_0|^pdT+\int_{T}|\tilde{\sigma}_0|^pdT\\
\leq&C\sum_{T\in {\cal T}_h}h_T^{1-p}h_T^{-1}\int_{T}|\tilde{\sigma}_0|^pdT+h_T^{1-p}\int_{\partial T}|\tilde{\sigma}_b|^pds
+h_T^{-p}\int_{T}|\tilde{\sigma}_0|^pdT+\int_{T}|\tilde{\sigma}_0|^pdT\\
\leq&C\sum_{T\in {\cal T}_h}h_T^{-p}\int_{T}|\tilde{\sigma}_0|^pdT+\sum_{e\subset {\cal E}_h \setminus \Gamma_+}h_T^{1-p}\int_{e}|\tilde{\sigma}_b|^pds\\
\leq&C\sum_{T\in {\cal T}_h}h_T^{-p}\int_{T}h_T^{qp}|Q_0(\sgn(\nabla\cdot(\bbeta v)+cv)|\nabla\cdot(\bbeta v)+cv|^{q-1})|^pdT\\
&+C\sum_{e\subset {\cal E}_h \setminus \Gamma_+}h_T^{1-p}\int_{e}h_T^p|Q_b(\sgn(\ljump\bbeta v\cdot \bn\rjump)|\ljump\bbeta v\cdot \bn\rjump|^{q-1})|^pds\\
\leq&C\sum_{T\in {\cal T}_h}h_T^{q}\int_{T}|\nabla\cdot(\bbeta v)+cv|^{q-1}\sgn(\nabla\cdot(\bbeta v)+cv)(\nabla\cdot(\bbeta v)+cv)dT\\
&+C\sum_{e\subset {\cal E}_h \setminus \Gamma_+}h_T\int_{e}|\ljump\bbeta v\cdot \bn\rjump|^{q-1}\sgn(\ljump\bbeta v\cdot \bn\rjump)\ljump\bbeta v\cdot \bn\rjump ds \\
\leq&C\|v\|_{M_h}^q,
\end{split}
\end{equation*}
which gives rise to the second inequality \eqref{inf-sup-2}. This completes the proof.
\end{proof}

\begin{theorem} \label{theoestimate-1} Let  $j=k-1$ or $j=k$ with $k\geq 1$. Let $u$ and  $(u_h; \lambda_h) \in M_{k-1,h}\times W_{j, h}^{0, \Gamma_+}$  be the exact solution of the first-order linear convection model problem \eqref{model} and the numerical solution arising from $L^p$-PDWG scheme (\ref{32})-(\ref{2}). The following error estimate holds true:
\begin{equation}\label{sestimate}
\3bar\epsilon_h\3bar\leq Ch^{qk/p}\|\nabla^{k}u\|^{q/p}_{L^q(\Omega)}.
\end{equation}
 \end{theorem}

\begin{proof} By letting $\sigma = \epsilon_h$  in (\ref{sehv}),  we have from (\ref{sehv2}) and \eqref{lu} that
\begin{equation}\label{EQ:April7:001_1}
\begin{split}
s(\epsilon_h, \epsilon_h)
= &\sum_{T\in {\cal T}_h}\langle \bbeta (u-{\cal Q}_h u)  \cdot \bn,\epsilon_b-\epsilon_0 \rangle_{\partial T} +   ({\cal Q}_hu-u, c\epsilon_0 )_T.
\end{split}
\end{equation}

For the first term on the right-hand side of \eqref{EQ:April7:001_1}, we use the Cauchy-Schwarz inequality to obtain
\begin{equation}\label{e1_1}
\begin{split}
&\Big|\sum_{T\in {\cal T}_h}\langle \bbeta (u-{\cal Q}_h u)  \cdot \bn,\epsilon_b-\epsilon_0 \rangle_{\partial T} \Big|\\
 \leq
 & C \Big(\sum_{T\in {\cal T}_h}\|  \bbeta(u-{\cal Q}_hu) \|^q_{L^q(\pT)}\Big)^{\frac{1}{q}} \Big(\sum_{T\in {\cal T}_h}\|\epsilon_b-\epsilon_0\|^p_{L^p(\pT)}\Big)^{\frac{1}{p}}.
\end{split}
\end{equation}
For the term $\| \bbeta(u-{\cal Q}_hu) \|^q_{L^q(\pT)}$, we have from the trace inequality \eqref{trace-inequality}  that
\begin{equation}\label{te2_1}
\begin{split}
\sum_{T\in {\cal T}_h}\| \bbeta(u-{\cal Q}_hu) \|^q_{L^q(\pT)}
\leq &  C \sum_{T\in {\cal T}_h}h_T^{-1}\Big(\| \bbeta(u-{\cal Q}_hu) \|^q_{L^q(T)}+h_T^q\|\nabla( \bbeta(u-{\cal Q}_hu) )\|^q_{L^q(T)}\Big)\\
 \leq & Ch^{kq-1}\|\nabla^{k}u\|^q_{L^q(\Omega)}.
\end{split}
\end{equation}
Substituting (\ref{te2_1}) into (\ref{e1_1}) and then using Young's inequality give
\begin{equation}\label{e3_1}
\begin{split}
&\Big|\sum_{T\in {\cal T}_h}\langle  \bbeta (u-{\cal Q}_h u)  \cdot \bn,  \epsilon_b-\epsilon_0\rangle_{\partial T}\Big|\\
\leq & Ch^{k}\|\nabla^{k}u\|_{L^q(\Omega)} \Big(\sum_{T\in {\cal T}_h}h_T^{1-p} \|\epsilon_b-\epsilon_0\|^p_{L^p(\pT)}\Big)^{\frac{1}{p}}\\
\leq & C_1  \sum_{T\in {\cal T}_h}  h_T^{1-p} \|\epsilon_b-\epsilon_0\|^p_{L^p(\pT)} + C_2 h^{qk}\|\nabla^{k}u\|^q_{L^q(\Omega)}.
\end{split}
\end{equation}

The second term can be analogously estimated by
\begin{equation}\label{e4:001_1}
\begin{split}
&\Big|\sum_{T\in {\cal T}_h}({\cal Q}_hu-u, c\epsilon_0 )_T\Big| \\
=&\Big|\sum_{T\in {\cal T}_h}({\cal Q}_hu-u, \bbeta \cdot\nabla \epsilon_0-c\epsilon_0 )_T\Big| \\
\leq &  \Big(\sum_{T\in {\cal T}_h}\| u-{\cal Q}_hu\|^q_{L^q(T)}\Big)^{\frac{1}{q}} \Big(\sum_{T\in {\cal T}_h}\| \bbeta \cdot\nabla \epsilon_0-c\epsilon_0\|^p_{L^p(T)}\Big)^{\frac{1}{p}}\\
\leq &  Ch^{k}\|\nabla^{k}u\|_{L^q(\Omega)}\Big(\sum_{T\in {\cal T}_h}\| \bbeta \cdot\nabla \epsilon_0-c\epsilon_0\|^p_{L^p(T)}\Big)^{\frac{1}{p}}\\
\leq & C_3\sum_{T\in {\cal T}_h}\| \bbeta \cdot\nabla \epsilon_0-c\epsilon_0\|^p_{L^p(T)}+C_4h^{qk}\|\nabla^{k}u\|^q_{L^q(\Omega)}.
\end{split}
\end{equation}
Substituting (\ref{e3_1}),  and \eqref{e4:001_1} into (\ref{EQ:April7:001_1}) gives
$$
(1-C_1-C_3)\3bar\epsilon_h\3bar^p\leq Ch^{qk}\|\nabla^{k}u\|^q_{L^q(\Omega)},
$$
which leads to
$$
\3bar\epsilon_h\3bar^p\leq Ch^{qk}\|\nabla^{k}u\|^q_{L^q(\Omega)},
$$
by choosing $C_i$ such that $1-C_1-C_3\ge C_0>0$. This completes the proof of the theorem.
\end{proof}

\begin{theorem}\label{erroreestimate}
 Let $j=k-1$ or $j=k$ with $k\geq 1$. Let $u$ and $(u_h; \lambda_h) \in M_{k-1,h}\times W_{j, h}^{0, \Gamma_+}$ be the numerical solutions of (\ref{model}) and (\ref{32})-(\ref{2}). Assume that the exact solution $u$ satisfies $\nabla^ku\in L^q(\Omega)$. Then, the following error estimate holds true:
$$\|e_h\|_{M_h}\lesssim h^{k}\|\nabla^ku\|_{L^q{(\Omega})}.$$
\end{theorem}

\begin{proof}
It follows from Lemma \ref{inf-sup-0} that for any $v\in M_{k-1,h}$, there exists a function $\tilde{\sigma}\in W_{j,h}^{0, \Gamma_+}$ such that
\begin{equation}\label{error-estimate-1}
|b(v,\tilde{\sigma})|=\|v\|_{M_h}^q,~~~~\3bar\tilde{\sigma}\3bar^p\leq C \|v\|_{M_h}^q.
\end{equation}
From Lemma \ref{errorequa} there holds
\begin{equation}\label{error-estimate-2}
b(e_h,\sigma)=\ell_u(\sigma)-s( \epsilon_h , \sigma),
\end{equation}
for any $\sigma\in W_{j,h}^{0, \Gamma_+}.$

By taking $\sigma=\tilde{\sigma}$ in \eqref{error-estimate-2} and $v=e_h$ in \eqref{error-estimate-1}, then we apply the triangular inequality, \eqref{lu},      \eqref{e3_1} and \eqref{e4:001_1} with $\epsilon_h=\tilde{\sigma}$, and  \eqref{error-estimate-1} to obtain
 \begin{equation}\label{error-estimate-3}
\begin{split}
\|e_h\|_{M_h}^q&=|b(e_h,\tilde{\sigma})|\\
&\leq|\ell_u(\tilde{\sigma})|+|s( \epsilon_h ,\tilde{\sigma})|\\
&\leq(C_1+C_3)\3bar\tilde{\sigma}\3bar^p+Ch^{qk}\|\nabla^ku\|_{L^q(\Omega)}^{q}+|s( \epsilon_h ,\tilde{\sigma})|\\
&\leq(C_1+C_3) \|e_h\|_{M_h}^q+Ch^{qk}\|\nabla^ku\|_{L^q(\Omega)}^{q}+|s( \epsilon_h ,\tilde{\sigma})|.
\end{split}
\end{equation}

Next, it suffices to estimate the last term $|s( \epsilon_h ,\tilde{\sigma})|$ on the right-hand side of \eqref{error-estimate-3}. Using the Cauchy-Schwarz inequality, \eqref{stabilizer-local} and Young's inequality, there holds
 \begin{equation}\label{error-estimate-4}
\begin{split}
&|s( \epsilon_h ,\tilde{\sigma})|=|\sum_{T\in {\cal T}_h}\rho h_T^{1-p}\int_{\partial T} |\epsilon_0-\epsilon_b|^{p-1}sgn(\epsilon_0-\epsilon_b)(\tilde{\sigma}_0-\tilde{\sigma}_b)ds\\
&+\tau\int_{ T}  |\bbeta\cdot\nabla\epsilon_0-c\epsilon_0|^{p-1}sgn(\bbeta\cdot\nabla\epsilon_0-c\epsilon_0)(\bbeta\cdot\nabla \tilde{\sigma}_0-c\tilde{\sigma}_0)dT|\\
\leq& C\Big(\sum_{T\in {\cal T}_h}\rho h_T^{1-p}\int_{\partial T} |\epsilon_0-\epsilon_b|^{(p-1)q}ds\Big)^{\frac{1}{q}} \Big(\sum_{T\in {\cal T}_h}\rho h_T^{1-p}\int_{\partial T}|\tilde{\sigma}_0-\tilde{\sigma}_b|^pds\Big)^{\frac{1}{p}}\\
&+C\Big(\sum_{T\in {\cal T}_h}\tau\int_{T} |\bbeta\cdot\nabla\epsilon_0-c\epsilon_0|^{(p-1)q}dT\Big)^{\frac{1}{q}} \Big(\sum_{T\in {\cal T}_h}\tau\int_{T}|\bbeta\cdot\nabla \tilde{\sigma}_0-c\tilde{\sigma}_0|^pdT\Big)^{\frac{1}{p}}\\
\leq&C\Big(\sum_{T\in {\cal T}_h}\rho h_T^{1-p}\int_{\partial T} |\epsilon_0-\epsilon_b|^{p}ds\Big)^{\frac{1}{q}}\3bar\tilde{\sigma}\3bar+
C\Big(\sum_{T\in {\cal T}_h}\tau\int_{T} |\bbeta\cdot\nabla\epsilon_0-c\epsilon_0|^{p}dT\Big)^{\frac{1}{q}}\3bar\tilde{\sigma}\3bar\\
\leq&C\3bar\epsilon_h\3bar^{\frac{p}{q}}\3bar\tilde{\sigma}\3bar\\
\leq&\frac{C}{q}\3bar\epsilon_h\3bar^p+\frac{C}{p}\3bar\tilde{\sigma}\3bar^p\\
\leq&Ch^{qk}\|\nabla^ku\|_{L^q{(\Omega})}^q+C_5  \|e_h\|_{M_h}^q,
\end{split}
\end{equation}
where we used estimate \eqref{sestimate}  and \eqref{error-estimate-1}  in the last step.

Substituting \eqref{error-estimate-4} into \eqref{error-estimate-3} gives rise to the error estimate
 $$\|e_h\|_{M_h}^q\leq(C_1+C_3+C_5) \|e_h\|_{M_h}^q+Ch^{qk}\|\nabla^ku\|_{L^q{(\Omega})}^q.$$
Choosing $C_i (i=1,2,3)$   such that $1-(C_1+C_3+C_5) >C_6>0$ leads to the desired estimate. This completes the proof of the theorem.
\end{proof}

\section{Numerical Experiments}\label{Section:Numerics}

This section presents some numerical results to demonstrate the efficiency and accuracy of the $L^p$-PDWG method. The iteration techniques in \cite{VO1996} shall be incorporated into the $L^p$-PDWG scheme \eqref{32}-\eqref{2} to solve the nonlinear system. Specifically, the lagged diffusivity fixed point iteration shall be extended to the $L^p$ optimization problem so that the nonlinear problem \eqref{32}-\eqref{2} can be expressed as a linear problem at each iteration.  To this end, given a numerical approximation at step $n$ denoted by $(u_h^{n};\lambda_h^{n})$, a updated numerical approximation  for $L^p$-PDWG scheme \eqref{32}-\eqref{2}  at step $n+1$  seeks $(u_h^{n+1};\lambda_h^{n+1})\in M_{k-1,h}\times W_{j,h}^{0,\Gamma_+}$ satisfying
\begin{eqnarray}
\widetilde{s}(\lambda_h^{n+1}, \sigma)+b(u_h^{n+1},\sigma)&=& \sum_{e\subset \Gamma_-}\langle \sigma_b, \bbeta \cdot \bn g  \rangle_{e}-(f,\sigma_0), \qquad \forall\sigma\in W_{j,h}^{0, \Gamma_+},\label{ErrorEstimate:Lq-1-1}\\
b(v,\lambda_h^{n+1})&=&0,\qquad \qquad\qquad\qquad \qquad \qquad~~~\forall v\in M_{k-1,h},\label{ErrorEstimate:Lq-1-2}
\end{eqnarray}
where
\begin{eqnarray*}
&&\widetilde{s}_T(\lambda_h^{n+1}, \sigma)=\rho h_T^{1-p}\int_{\partial T}(|\lambda_0^n-\lambda_b^n|+\varepsilon)^{p-2}(\lambda_0^{n+1}-\lambda_b^{n+1})(\sigma_0-\sigma_b)ds\\
& & +\tau   \int_{ T} (|\bbeta\cdot\nabla\lambda_0^n-c\lambda_0^n|+\varepsilon)^{p-2}(\bbeta\cdot\nabla\lambda_0^{n+1}-c\lambda_0^{n+1})(\bbeta\cdot\nabla\sigma_0-c\sigma_0)dT,
\end{eqnarray*}
where $\varepsilon$ is a small and positive constant. We choose $\varepsilon=10^{-4}$ in all numerical tests.
The numerical iterative procedure will automatically stop when $\max\{|u_h^{n+1}-u_h^{n}|, |\lambda_h^{n+1}-\lambda_h^{n}|\} \leq 10^{-5}$ is satisfied.

The following metrics are employed to measure the error functions:
\begin{equation*}
\begin{split}
&\|e_h\|_{0,q}=\Big(\sum_{T\in {\cal T}_h}\int_T|e_h|^qdT\Big)^{1/q},\qquad~
\|\epsilon_0\|_{0,p}=\Big(\sum_{T\in {\cal T}_h}\int_T|\epsilon_0|^pdT\Big)^{1/p},\\
&\|\epsilon_b\|_{0,p}=\Big(\sum_{T\in {\cal T}_h}h_T\int_{\partial T}|\epsilon_b|^pds\Big)^{1/p},\quad
\|\epsilon_0\|_{1,p}=\Big(\sum_{T\in {\cal T}_h}\int_T|\nabla\epsilon_0|^pdT\Big)^{1/p},\\
&\|\epsilon_0\|_{2,p}=\Big(\sum_{T\in {\cal T}_h}\int_T|\Delta\epsilon_0|^pdT\Big)^{1/p}.
\end{split}
\end{equation*}

Our numerical experiments are conducted on two types of polygonal domains: (1) an unit square domain $\Omega_1=[0,1]^2$, (2) an L-shaped domain $\Omega_2$ with vertices $(0,0)$, $(0.5,0)$, $(0.5,0.5)$, $(1,0.5)$, $(1,1)$, and $(0,1)$. The finite element partition ${\cal T}_h$  is obtained by uniformly partitioning a coarse triangulation of the domain  which refines each triangular element into four sub-triangles by connecting the mid-points of the three edges of the triangular element. The parameters $\rho>0$ and $\tau\geq 0$   for each text example are specified in each table/figure.

\subsection{Continuous convection $\bbeta$}

Table \ref{NE1} illustrates the numerical performance of the $L^p$-PDWG scheme \eqref{ErrorEstimate:Lq-1-1}-\eqref{ErrorEstimate:Lq-1-2} for $k=1$ and $j=k$ for different $p$  with the exact solution $u=\cos(\pi x)\cos(\pi y)$ on the  domain $\Omega_1$. The convection vector is  $\bbeta=[-y,x]$, and the reaction coefficient is $c=1$. These numerical results suggest that (1) for cases of $p=2, 3, 5$, the optimal convergence order for $e_h$ in $L^q$ norm is observed; (2) for cases of $p=1.2, 1.6$, the numerical  convergence rate for $\|e_h\|_{0,q}$ seems to suffer, which might be caused by the large value of $(|\lambda_0^n-\lambda_b^n|+\varepsilon)^{p-2}$ in the stabilizer $\widetilde{s}_T(\lambda_h^{n+1}, \sigma)$; (3)  the convergence rates for $\|\epsilon_0\|_{m,p}$ and $\|\epsilon_b\|_{0,p}$ are of order ${\cal O} (h^{p-m})$ for $m=0,1$ and ${\cal O} (h^{p})$, respectively; (4) for cases of $p=3, 5$, the numerical performance is improved for sufficiently large stabilization parameters.

\begin{table}[htbp]
\tiny
\caption{Numerical errors and convergence rates for $L^p$-PDWG method with $k=1$ and $j=k$.}\label{NE1}
\begin{tabular}{p{0.4cm}p{1.34cm}p{0.80cm}p{1.43cm}p{0.8cm}p{1.43cm}p{0.8cm}p{1.34cm}p{0.8cm}}
\hline
$1/h$&$\|e_h\|_{0,q}$&Rate&$\|\epsilon_0\|_{0,p}$&Rate&$\|\epsilon_b\|_{0,p}$&Rate&$\|\epsilon_0\|_{1,p}$&Rate\\
\hline
&\text{$p=1.2,$~$(\rho,\tau)=(1,1)$}&&&&&&&\\
\hline
8   &1.89e-1    &0.55    &1.63e-5    &1.69    &4.02e-4    &2.01    &3.68e-4    &1.05\\
16  &1.51e-1    &0.32    &6.01e-6    &1.44    &1.27e-4    &1.66    &2.18e-4    &0.76\\
32  &1.21e-1    &0.32    &2.46e-6    &1.29    &4.82e-5    &1.40    &1.60e-4    &0.44\\
64  &9.67e-2    &0.33    &1.05e-6    &1.23    &1.98e-5    &1.29    &1.30e-4    &0.30\\
\hline
&\text{$p=1.6,$~$(\rho,\tau)=(1,1)$}&&&&&&&\\
\hline
8   &1.09e-1    &0.88    &7.41e-4    &2.64    &9.30e-3    &2.86    &1.68e-2    &1.97\\
16  &6.11e-2    &0.84    &1.33e-4    &2.48    &1.40e-3    &2.73    &4.62e-3    &1.86\\
32  &3.47e-2    &0.82    &3.20e-5    &2.05    &3.01e-4    &2.22    &1.91e-3    &1.27\\
64  &1.99e-2    &0.80    &9.38e-6    &1.77    &8.32e-5    &1.85    &1.04e-3    &0.88\\
\hline
&\text{$p=2,$~$(\rho,\tau)=(1,1)$}&&&&&&&\\
\hline
8   &9.35e-2    &0.95    &5.34e-3    &2.19    &4.11e-2    &2.37    &1.16e-1    &1.49\\
16  &4.82e-2    &0.96    &1.24e-3    &2.11    &8.24e-3    &2.32    &4.22e-2    &1.46\\
32  &2.48e-2    &0.96    &3.02e-4    &2.04    &1.82e-3    &2.18    &1.73e-2    &1.28\\
64  &1.27e-2    &0.96    &7.49e-5    &2.01    &4.27e-4    &2.09    &7.79e-3    &1.15\\
\hline
&\text{$p=3,$~$(\rho,\tau)=(1e+4,1e+3)$}&&&&&&&\\
\hline
8   &8.01e-2    &1.00    &2.36e-4    &2.02    &1.05e-3    &2.06    &5.41e-3    &1.13\\
16  &4.00e-2    &1.00    &5.77e-5    &2.03    &2.34e-4    &2.17    &2.28e-3    &1.25\\
32  &2.00e-2    &1.00    &1.07e-5    &2.43    &3.99e-5    &2.55    &7.22e-4    &1.66\\
64  &1.00e-2    &1.00    &1.49e-6    &2.84    &5.20e-6    &2.94    &1.72e-4    &2.07\\
\hline
&\text{$p=5,$~$(\rho,\tau)=(1e+13,1e+12)$}&&&&&&&\\
\hline
8   &7.24e-2    &1.03    &2.84e-6    &4.06    &8.66e-6    &3.98    &7.29e-5    &3.00\\
16  &3.58e-2    &1.01    &8.46e-8    &5.07    &2.55e-7    &5.09    &4.30e-6    &4.09\\
32  &1.78e-2    &1.01    &2.36e-9    &5.17    &6.92e-9    &5.20    &2.30e-7    &4.23\\
64  &8.89e-3    &1.00    &6.76e-11   &5.12    &1.91e-10   &5.18    &1.23e-8    &4.22\\
\hline
\end{tabular}
\end{table}

Tables \ref{NE2-1}-\ref{NE2-2} show the numerical results for the $L^p$-PDWG numerical approximation $(u_h;\lambda_h)$ on the domain $\Omega_1$ for $k=1$ and $j=k$ for different values of $p$. The convection vector is $\bbeta=[y-0.5,0.5-x]$, and the reaction coefficient is $c=1$. The exact solution is $u=\cos(\pi x)\cos(\pi y)$. We can observe that the optimal order of convergence for $e_h$ in $L^q$ norm is ${\cal O}(h)$, and the convergence rates  for $\|\epsilon_0\|_{0,p}$, $\|\epsilon_b\|_{0,p}$ and $\|\epsilon_0\|_{1,p}$ are of orders  ${\cal O}(h^p)$, ${\cal O}(h^p)$ and ${\cal O}(h^{p-1})$, respectively. It should be pointed out that although the developed theory is not applicable to the case of $j=k$ and $\tau=0$, we still have observed the good numerical performance. It is interesting to note that for cases of $p=1.2, 1.6$,  the numerical performance of $\|e_h\|_{0,q}$ shown in Table \ref{NE2-2} is much better than that in Table \ref{NE1}. This implies that the choice of the convection vector $\bbeta$ may have an effect on the convergence rate of $\|e_h\|_{0,q}$.

\begin{table}[htbp]
\tiny
\caption{Numerical errors and convergence rates for $L^p$-PDWG method with $k=1$ and $j=k$.}\label{NE2-1}
\begin{tabular}{p{0.4cm}p{1.34cm}p{0.80cm}p{1.34cm}p{0.8cm}p{1.34cm}p{0.8cm}p{1.34cm}p{0.8cm}}
\hline
$1/h$&$\|e_h\|_{0,q}$&Rate&$\|\epsilon_0\|_{0,p}$&Rate&$\|\epsilon_b\|_{0,p}$&Rate&$\|\epsilon_0\|_{1,p}$&Rate\\
\hline
&\text{$p=1.2,$~$(\rho,\tau)=(1,0)$}&&&&&&&\\
\hline
8   &1.21e-1    &1.01    &6.77e-4    &1.79    &2.12e-2    &1.90    &1.36e-2    &1.40\\
16  &5.97e-2    &1.02    &2.71e-4    &1.32    &8.12e-3    &1.39    &8.38e-3    &0.70\\
32  &2.98e-2    &1.00    &1.16e-4    &1.22    &3.40e-3    &1.25    &6.48e-3    &0.37\\
64  &1.49e-2    &1.00    &5.05e-5    &1.20    &1.46e-3    &1.22    &5.41e-3    &0.26\\
\hline
&\text{$p=1.6,$~$(\rho,\tau)=(1,0)$}&&&&&&&\\
\hline
8   &7.94e-2    &1.16    &2.19e-3    &2.40    &3.17e-2    &2.47    &4.11e-2    &1.90\\
16  &3.89e-2    &1.03    &6.44e-4    &1.76    &8.99e-3    &1.82    &1.92e-2    &1.10\\
32  &1.94e-2    &1.01    &2.08e-4    &1.63    &2.84e-3    &1.66    &1.11e-2    &0.79\\
64  &9.66e-3    &1.00    &6.87e-5    &1.60    &9.24e-4    &1.62    &6.97e-3    &0.67\\
\hline
&\text{$p=2,$~$(\rho,\tau)=(1,0)$}&&&&&&&\\
\hline
8   &6.99e-2    &1.19    &6.77e-3    &2.62    &6.19e-2    &2.66    &1.23e-1    &2.03\\
16  &3.42e-2    &1.03    &1.52e-3    &2.15    &1.35e-2    &2.20    &4.50e-2    &1.45\\
32  &1.70e-2    &1.01    &3.72e-4    &2.03    &3.21e-3    &2.07    &1.96e-2    &1.20\\
64  &8.49e-3    &1.00    &9.30e-5    &2.00    &7.91e-4    &2.02    &9.22e-3    &1.09\\
\hline
&\text{$p=3,$~$(\rho,\tau)=(1e+4,0)$}&&&&&&&\\
\hline
8   &6.29e-2    &1.25    &1.22e-5    &3.67    &6.00e-5    &3.66    &2.11e-4    &2.94\\
16  &3.07e-2    &1.04    &1.33e-6    &3.20    &6.33e-6    &3.24    &3.95e-5    &2.42\\
32  &1.52e-2    &1.01    &1.59e-7    &3.06    &7.40e-7    &3.10    &8.31e-6    &2.24\\
64  &7.60e-3    &1.00    &1.98e-8    &3.01    &9.01e-8    &3.04    &1.89e-6    &2.13\\
\hline
&\text{$p=5,$~$(\rho,\tau)=(1e+5,0)$}&&&&&&&\\
\hline
8   &5.99e-2    &1.18    &4.17e-4    &3.78    &1.21e-3    &3.85    &6.92e-3    &3.24\\
16  &2.89e-2    &1.05    &1.06e-5    &5.30    &3.03e-5    &5.32    &3.46e-4    &4.32\\
32  &1.43e-2    &1.01    &2.96e-7    &5.16    &8.20e-7    &5.21    &1.82e-5    &4.25\\
64  &7.16e-3    &1.00    &8.74e-9    &5.08    &2.35e-8    &5.13    &9.81e-7    &4.21\\
\hline
\end{tabular}
\end{table}

\begin{table}[htbp]
\tiny
\caption{Numerical errors and convergence rates for $L^p$-PDWG method with $k=1$ and $j=k$.}\label{NE2-2}
\begin{tabular}{p{0.4cm}p{1.34cm}p{0.80cm}p{1.34cm}p{0.8cm}p{1.34cm}p{0.8cm}p{1.34cm}p{0.8cm}}
\hline
$1/h$&$\|e_h\|_{0,q}$&Rate&$\|\epsilon_0\|_{0,p}$&Rate&$\|\epsilon_b\|_{0,p}$&Rate&$\|\epsilon_0\|_{1,p}$&Rate\\
\hline
&\text{$p=1.2,$~$(\rho,\tau)=(1,1)$}&&&&&&&\\
\hline
8   &1.25e-1    &1.06    &6.41e-4    &1.34    &2.01e-2    &1.47    &1.19e-2    &0.96\\
16  &6.17e-2    &1.02    &2.66e-4    &1.27    &7.96e-3    &1.33    &7.81e-3    &0.61\\
32  &3.07e-2    &1.01    &1.15e-4    &1.21    &3.37e-3    &1.24    &6.16e-3    &0.34\\
64  &1.53e-2    &1.00    &5.04e-5    &1.19    &1.46e-3    &1.21    &5.19e-3    &0.25\\
\hline
&\text{$p=1.6,$~$(\rho,\tau)=(1,1)$}&&&&&&&\\
\hline
8   &8.02e-2    &1.25    &2.12e-3    &2.07    &3.07e-2    &2.17    &3.79e-2    &1.56\\
16  &3.93e-2    &1.03    &6.38e-4    &1.73    &8.87e-3    &1.79    &1.83e-2    &1.05\\
32  &1.95e-2    &1.01    &2.08e-4    &1.62    &2.82e-3    &1.65    &1.08e-2    &0.76\\
64  &9.74e-3    &1.00    &6.86e-5    &1.60    &9.21e-4    &1.61    &6.82e-3    &0.66\\
\hline
&\text{$p=2,$~$(\rho,\tau)=(1,1)$}&&&&&&&\\
\hline
8   &7.00e-2    &1.23    &6.64e-3    &2.39    &6.05e-2    &2.46    &1.17e-1    &1.80\\
16  &3.42e-2    &1.03    &1.51e-3    &2.13    &1.33e-2    &2.18    &4.33e-2    &1.43\\
32  &1.70e-2    &1.01    &3.71e-4    &2.02    &3.19e-3    &2.06    &1.91e-2    &1.18\\
64  &8.49e-3    &1.00    &9.29e-5    &2.00    &7.89e-4    &2.02    &9.11e-3    &1.07\\
\hline
&\text{$p=3,$~$(\rho,\tau)=(1e+4,1e+3)$}&&&&&&&\\
\hline
8   &6.28e-2    &1.25    &1.22e-5    &3.66    &5.99e-5    &3.66    &2.10e-4    &2.93\\
16  &3.07e-2    &1.04    &1.33e-6    &3.20    &6.33e-6    &3.24    &3.93e-5    &2.42\\
32  &1.52e-2    &1.01    &1.59e-7    &3.05    &7.40e-7    &3.10    &8.30e-5    &2.24\\
64  &7.60e-3    &1.00    &1.98e-8    &3.01    &9.01e-8    &3.04    &1.89e-6    &2.13\\
\hline
&\text{$p=5,$~$(\rho,\tau)=(1e+5,1e+4)$}&&&&&&&\\
\hline
8   &5.99e-2    &1.18    &4.17e-4    &3.76    &1.21e-3    &3.84    &6.92e-3    &3.22\\
16  &2.89e-2    &1.05    &1.06e-5    &5.30    &3.03e-5    &5.32    &3.46e-4    &4.32\\
32  &1.43e-2    &1.01    &2.96e-7    &5.16    &8.20e-7    &5.21    &1.82e-5    &4.25\\
64  &7.16e-3    &1.00    &8.74e-9    &5.08    &2.35e-8    &5.13    &9.81e-7    &4.21\\
\hline
\end{tabular}
\end{table}

Table \ref{NE3} demonstrates the numerical performance of the $L^p$-PDWG scheme on the domain $\Omega_1$ with $k=2$ and $j=k$ for different values of $p$. The exact solution is $u=\cos(x)\sin(y)$. The convection vector is $\bbeta=[1,1]$, and the reaction coefficient is $c=-1$. We can observe from Table \ref{NE3} that an optimal order of convergence can be obtained for $\|e_h\|_{0,q}$, while the convergence rates for $\|\epsilon_0\|_{m,p}$ and $\|\epsilon_b\|_{0,p}$  are of orders ${\cal O}(h^{p+1-m})$ for $m=0,1,2$,  and ${\cal O}(h^{p+1})$, respectively.

\begin{table}[htbp]
\tiny
\caption{Numerical errors and convergence rates for $L^p$-PDWG method with $k=2$ and $j=k$.}\label{NE3}
\begin{tabular}{p{0.2cm}p{1.15cm}p{0.4cm}p{1.28cm}p{0.4cm}p{1.28cm}p{0.4cm}p{1.28cm}p{0.4cm}p{1.1cm}p{0.5cm}}
\hline
$1/h$&$\|e_h\|_{0,q}$&Rate&$\|\epsilon_0\|_{0,p}$&Rate&$\|\epsilon_b\|_{0,p}$&Rate&$\|\epsilon_0\|_{1,p}$&Rate&$\|\epsilon_0\|_{2,p}$&Rate\\
\hline
&\text{$p=1.2,$~$(\rho,\tau)=(1,1)$}&&&&&&&&&\\
\hline
8   &7.14e-4    &1.96    &3.71e-7    &2.12    &9.27e-7    &2.23    &1.17e-5    &1.12    &1.26e-4    &0.01\\
16  &1.76e-4    &2.02    &8.29e-8    &2.16    &1.96e-7    &2.24    &5.34e-6    &1.13    &1.18e-4    &0.10\\
32  &4.35e-5    &2.02    &1.83e-8    &2.18    &4.21e-8    &2.22    &2.38e-6    &1.16    &1.06e-4    &0.15\\
64  &1.08e-5    &2.01    &4.01e-9    &2.19    &9.11e-9    &2.21    &1.05e-6    &1.18    &9.34e-5    &0.18\\
\hline
&\text{$p=1.6,$~$(\rho,\tau)=(1,1)$}&&&&&&&&&\\
\hline
8   &6.00e-4    &2.00    &7.03e-6    &2.69    &1.38e-5    &2.77    &2.07e-4    &1.71    &2.09e-3    &0.65\\
16  &1.50e-4    &2.00    &1.16e-6    &2.60    &2.22e-6    &2.64    &6.96e-5    &1.57    &1.42e-3    &0.56\\
32  &3.75e-5    &2.00    &1.94e-7    &2.58    &3.66e-7    &2.60    &2.34e-5    &1.57    &9.62e-4    &0.56\\
64  &9.39e-6    &2.00    &3.22e-8    &2.59    &6.05e-8    &2.60    &7.80e-6    &1.58    &6.44e-4    &0.58\\
\hline
&\text{$p=2,$~$(\rho,\tau)=(1,1)$}&&&&&&&&&\\
\hline
8   &5.82e-4    &2.01    &1.32e-4    &2.94    &2.24e-4    &2.99    &3.83e-3    &1.95    &3.59e-2    &0.92\\
16  &1.46e-4    &2.01    &1.69e-5    &2.97    &2.83e-5    &2.99    &9.89e-4    &1.95    &1.86e-2    &0.95\\
32  &3.60e-5    &2.00    &2.15e-6    &2.98    &3.58e-6    &2.98    &2.52e-4    &1.97    &9.48e-3    &0.97\\
64  &8.99e-6    &2.00    &2.71e-7    &2.98    &4.51e-7    &2.99    &6.38e-5    &1.98    &4.78e-3    &0.99\\
\hline
&\text{$p=3,$~$(\rho,\tau)=(1e+4,1e+3)$}&&&&&&&&&\\
\hline
8   &4.84e-4    &2.12    &2.08e-5    &2.94    &2.82e-5    &2.99    &6.14e-4    &1.90    &4.23e-3    &1.01\\
16  &1.18e-4    &2.04    &1.63e-6    &3.67    &2.16e-6    &3.70    &1.00e-4    &2.61    &1.28e-3    &1.73\\
32  &2.94e-5    &2.00    &1.06e-7    &3.94    &1.40e-7    &3.95    &1.31e-5    &2.93    &3.07e-4    &2.06\\
64  &7.36e-6    &2.00    &6.71e-9    &3.98    &8.78e-9    &3.99    &1.66e-6    &2.99    &7.01e-5    &2.13\\
\hline
&\text{$p=5,$~$(\rho,\tau)=(1e+12,1e+11)$}&&&&&&&&&\\
\hline
8   &4.88e-4    &2.12    &8.50e-7    &4.69    &9.74e-7    &4.71    &3.30e-5    &3.39    &2.68e-4    &2.17\\
16  &1.23e-4    &1.99    &1.30e-8    &6.03    &1.48e-8    &6.04    &1.01e-6    &5.03    &1.55e-5    &4.11\\
32  &3.09e-5    &1.99    &2.00e-10   &6.03    &2.24e-10   &6.04    &3.06e-8    &5.05    &8.37e-7    &4.21\\
64  &7.73e-6    &2.00    &3.11e-12   &6.01    &3.47e-12   &6.01    &9.44e-10   &5.02    &4.53e-8    &4.21\\
\hline
\end{tabular}
\end{table}

Table \ref{NE4} presents the numerical results of $L^p$-PDWG method  when the exact solution  is given by $u=\sin(\pi x)\cos(\pi y)$ on the domain $\Omega_1$. We take $k=2$ and $j=k-1$ for  different values of $p$. The convection vector is $\bbeta=[1,-1]$, and the reaction coefficient is $c=1$. These numerical results indicate that the convergence rate for $\|e_h\|_{0,q}$ is of an optimal rate   ${\cal O}(h^{2})$, and the convergence rates for $\|\epsilon_0\|_{m,p}$ and $\|\epsilon_b\|_{0,p}$ are of orders ${\cal O}(h^{p+1-m})$ for $m=0,1$ and ${\cal O}(h^{p+1})$, respectively.

\begin{table}[htbp]
\tiny
\caption{Numerical errors and convergence rates for $L^p$-PDWG method with $k=2$ and $j=k-1$.}\label{NE4}
\begin{tabular}{p{0.4cm}p{1.34cm}p{0.80cm}p{1.43cm}p{0.8cm}p{1.43cm}p{0.8cm}p{1.43cm}p{0.8cm}}
\hline
$1/h$&$\|e_h\|_{0,q}$&Rate&$\|\epsilon_0\|_{0,p}$&Rate&$\|\epsilon_b\|_{0,p}$&Rate&$\|\epsilon_0\|_{1,p}$&Rate\\
\hline
&\text{$p=1.2,$~$(\rho,\tau)=(1,0)$}&&&&&&&\\
\hline
8   &9.45e-3    &1.97    &1.90e-6    &2.39    &3.35e-5    &2.39    &7.76e-5    &1.40\\
16  &2.38e-3    &1.99    &4.01e-7    &2.24    &7.10e-6    &2.24    &3.28e-5    &1.24\\
32  &5.95e-4    &2.00    &8.68e-8    &2.21    &1.54e-6    &2.21    &1.42e-5    &1.21\\
64  &1.49e-4    &2.00    &1.89e-8    &2.20    &3.34e-7    &2.20    &6.17e-6    &1.20\\
\hline
&\text{$p=1.6,$~$(\rho,\tau)=(1e+1,0)$}&&&&&&&\\
\hline
8   &6.03e-3    &1.98    &3.73e-6    &2.81    &3.51e-5    &2.81    &1.21e-4    &1.82\\
16  &1.51e-3    &1.99    &6.00e-7    &2.64    &5.65e-6    &2.64    &3.89e-5    &1.64\\
32  &3.79e-4    &2.00    &9.85e-8    &2.61    &9.28e-7    &2.61    &1.28e-5    &1.61\\
64  &9.47e-5    &2.00    &1.62e-8    &2.60    &1.53e-7    &2.60    &4.21e-6    &1.60\\
\hline
&\text{$p=2,$~$(\rho,\tau)=(1,0)$}&&&&&&&\\
\hline
8   &5.16e-3    &1.98    &6.93e-4    &2.98    &4.44e-3    &2.98    &1.92e-2    &1.99\\
16  &1.29e-3    &2.00    &8.69e-5    &3.00    &5.57e-4    &3.00    &4.82e-3    &2.00\\
32  &3.24e-4    &2.00    &1.09e-5    &3.00    &6.97e-5    &3.00    &1.20e-3    &2.00\\
64  &8.09e-5    &2.00    &1.36e-6    &3.00    &8.71e-6    &3.00    &3.01e-4    &2.00\\
\hline
&\text{$p=3,$~$(\rho,\tau)=(1e+4,0)$}&&&&&&&\\
\hline
8   &4.47e-3    &1.97    &8.95e-6    &3.00    &3.39e-5    &3.00    &1.95e-4    &2.02\\
16  &1.12e-3    &1.99    &6.48e-7    &3.79    &2.45e-6    &3.79    &2.82e-5    &2.79\\
32  &2.82e-4    &2.00    &4.10e-8    &3.98    &1.55e-7    &3.98    &3.56e-6    &2.98\\
64  &7.04e-5    &2.00    &2.56e-9    &4.00    &9.71e-9    &4.00    &4.46e-7    &3.00\\
\hline
&\text{$p=5,$~$(\rho,\tau)=(1e+13,0)$}&&&&&&&\\
\hline
8   &4.15e-3    &1.97    &2.03e-7    &5.44    &4.89e-7    &5.44    &3.45e-6    &4.46\\
16  &1.04e-3    &2.00    &3.22e-9    &5.98    &7.74e-9    &5.98    &1.09e-7    &4.99\\
32  &2.60e-4    &2.00    &5.03e-11   &6.00    &1.21e-10   &6.00    &3.40e-9    &5.00\\
64  &6.51e-5    &2.00    &7.86e-13   &6.00    &1.89e-12   &6.00    &1.06e-10   &5.00\\
\hline
\end{tabular}
\end{table}

 Table \ref{NE5} reports the numerical results for the exact solution $u=x(1-x)y(1-y)(y-0.25)^2$ on the L-shaped domain $\Omega_2$ with $k=2$ and $j=k-1$ for  different values of $p$. The convection vector is $\bbeta=[y-0.5,-x+0.25]$, and the reaction coefficient is $c=1$. These numerical results imply that (1) for cases of $p=2, 3, 5$, the convergence rate  for $\|e_h\|_{0,q}$ is of an optimal order ${\cal O}(h^{2})$; (2) for the cases of  $p=1.2, 1.6$, the convergence rate for $\|e_h\|_{0,q}$ seems to be of an order less than ${\cal O}(h^{2})$. We conjecture this deterioration might be caused by the large value of $(|\lambda_0^n-\lambda_b^n|+\varepsilon)^{p-2}$   in the stabilizer $\widetilde{s}_T(\lambda_h^{n+1}, \sigma)$; (3) the convergence rates for  $\|\epsilon_0\|_{m,p}$ and $\|\epsilon_b\|_{0,p}$ are of orders ${\cal O}(h^{p+1-m})$ for $m=0,1$ and ${\cal O}(h^{p+1})$, respectively.

\begin{table}[htbp]
\tiny
\caption{Numerical errors and convergence rates for $L^p$-PDWG method with $k=2$ and {\color{green}$j=k-1$}.}\label{NE5}
\begin{tabular}{p{0.4cm}p{1.34cm}p{0.80cm}p{1.43cm}p{0.8cm}p{1.43cm}p{0.8cm}p{1.43cm}p{0.8cm}}
\hline
$1/h$&$\|e_h\|_{0,q}$&Rate&$\|\epsilon_0\|_{0,p}$&Rate&$\|\epsilon_b\|_{0,p}$&Rate&$\|\epsilon_0\|_{1,p}$&Rate\\
\hline
&\text{$p=1.2,$~$(\rho,\tau)=(1,0)$}&&&&&&&\\
\hline
8   &4.61e-4    &1.94    &1.07e-8    &2.05    &2.78e-7    &1.93    &4.15e-7    &0.99\\
16  &1.25e-4    &1.89    &2.37e-9    &2.18    &6.18e-8    &2.17    &1.83e-7    &1.18\\
32  &4.16e-5    &1.58    &5.24e-10   &2.18    &1.38e-8    &2.17    &8.02e-8    &1.19\\
64  &1.64e-5    &1.34    &1.15e-10   &2.19    &3.03e-9    &2.18    &3.51e-8    &1.19\\
\hline
&\text{$p=1.6,$~$(\rho,\tau)=(1,0)$}&&&&&&&\\
\hline
8   &2.46e-4    &1.91    &2.66e-7    &2.44    &3.18e-6    &2.36    &8.66e-6    &1.42\\
16  &6.44e-5    &1.93    &4.48e-8    &2.57    &5.39e-7    &2.56    &2.86e-6    &1.60\\
32  &1.68e-5    &1.94    &7.49e-9    &2.58    &9.06e-8    &2.57    &9.48e-7    &1.60\\
64  &4.44e-6    &1.92    &1.25e-9    &2.59    &1.51e-8    &2.59    &3.13e-7    &1.60\\
\hline
&\text{$p=2,$~$(\rho,\tau)=(1,0)$}&&&&&&&\\
\hline
8   &1.91e-4    &1.88    &6.03e-6    &2.76    &4.63e-5    &2.70    &1.75e-4    &1.79\\
16  &4.96e-5    &1.94    &7.80e-7    &2.95    &6.05e-6    &2.94    &4.40e-5    &1.99\\
32  &1.26e-5    &1.98    &9.91e-8    &2.98    &7.70e-7    &2.97    &1.10e-5    &2.00\\
64  &3.18e-6    &1.99    &1.25e-8    &2.99    &9.70e-8    &2.99    &2.75e-6    &2.00\\
\hline
&\text{$p=3,$~$(\rho,\tau)=(1e+2,0)$}&&&&&&&\\
\hline
8   &1.47e-4    &1.83    &4.78e-5    &2.27    &2.29e-4    &2.13    &1.16e-3    &1.38\\
16  &3.81e-5    &1.94    &6.39e-6    &2.90    &3.03e-5    &2.92    &2.96e-4    &1.97\\
32  &9.57e-6    &1.99    &5.00e-7    &3.68    &2.32e-6    &3.70    &4.50e-5    &2.71\\
64  &2.38e-6    &2.01    &3.21e-8    &3.96    &1.49e-7    &3.97    &5.72e-6    &2.98\\
\hline
&\text{$p=5,$~$(\rho,\tau)=(1e+11,0)$}&&&&&&&\\
\hline
8   &1.25e-4    &1.80    &3.21e-7    &5.00    &1.13e-6    &4.64    &6.23e-6    &4.17\\
16  &3.29e-5    &1.92    &5.45e-9    &5.88    &2.02e-8    &5.81    &2.04e-7    &4.93\\
32  &8.24e-6    &2.00    &8.81e-11   &5.95    &3.24e-10   &5.96    &6.40e-9    &4.99\\
64  &2.04e-6    &2.02    &1.40e-12   &5.97    &5.09e-12   &5.99    &2.00e-10   &5.00\\
\hline
\end{tabular}
\end{table}

\subsection{Discontinuous convection $\bbeta$}

Tables \ref{NE6-1}-\ref{NE6-2} present some numerical results for the numerical approximation $(u_h;\lambda_h)$ on the square domain $\Omega_1$ for different values of $p$. The exact solution is given by $u=\sin(x)\cos(y)$. The convection vector is piece-wisely defined in the sense that $\bbeta=[y+1,-x-1]$ for $y<1-x$ and $\bbeta=[y-2,2-x]$ otherwise, and the reaction coefficient is $c=-1$. We take $k=2$ and $j=k-1$. We can observe that for the cases of $p\geq1.6$, the convergence rate for $e_h$ in $L^q$ norm reaches an optimal order of ${\cal O}(h^{2})$ while for the case of $p=1.2$, $\|e_h\|_{0,q}$ seems to retain a convergence of order $r=1.7$. For the dual variable $\lambda_h$, the numerical rates of convergence for $\|\epsilon_0\|_{m,p}$ and $\|\epsilon_b\|_{0,p}$ remain to be of order $r=p+1-m$ for $m=0,1$ and $r=p+1$, respectively.

\begin{table}[htbp]
\tiny
\caption{Numerical errors and convergence rates for $L^p$-PDWG method with $k=2$ and $j=k-1$.}\label{NE6-1}
\begin{tabular}{p{0.4cm}p{1.34cm}p{0.80cm}p{1.43cm}p{0.8cm}p{1.43cm}p{0.8cm}p{1.34cm}p{0.8cm}}
\hline
$1/h$&$\|e_h\|_{0,q}$&Rate&$\|\epsilon_0\|_{0,p}$&Rate&$\|\epsilon_b\|_{0,p}$&Rate&$\|\epsilon_0\|_{1,p}$&Rate\\
\hline
&\text{$p=1.2,$~$(\rho,\tau)=(1,0)$}&&&&&&&\\
\hline
8   &1.26e-3    &1.97    &1.50e-7    &2.06    &5.24e-6    &1.98    &5.37e-6    &1.05\\
16  &3.53e-4    &1.83    &3.41e-8    &2.15    &1.21e-6    &2.11    &2.42e-6    &1.15\\
32  &1.05e-4    &1.75    &7.58e-9    &2.17    &2.72e-7    &2.16    &1.07e-6    &1.18\\
64  &3.21e-5    &1.71    &1.67e-9    &2.19    &6.01e-8    &2.18    &4.68e-7    &1.19\\
\hline
&\text{$p=1.6,$~$(\rho,\tau)=(1,0)$}&&&&&&&\\
\hline
8   &7.69e-4    &2.08    &3.04e-6    &2.65    &4.89e-5    &2.61    &9.36e-5    &1.64\\
16  &1.89e-4    &2.02    &5.08e-7    &2.58    &8.23e-6    &2.56    &3.10e-5    &1.59\\
32  &4.78e-5    &1.99    &8.50e-8    &2.58    &1.39e-6    &2.57    &1.03e-5    &1.59\\
64  &1.22e-5    &1.98    &1.42e-8    &2.59    &2.31e-7    &2.58    &3.42e-6    &1.59\\
\hline
&\text{$p=2,$~$(\rho,\tau)=(1,0)$}&&&&&&&\\
\hline
8   &6.57e-4    &2.08    &5.71e-5    &2.86    &5.88e-4    &2.84    &1.61e-3    &1.83\\
16  &1.59e-4    &2.04    &7.42e-6    &2.94    &7.67e-5    &2.94    &4.15e-4    &1.95\\
32  &3.94e-5    &2.02    &9.45e-7    &2.97    &9.77e-6    &2.97    &1.05e-4    &1.98\\
64  &9.80e-6    &2.01    &1.19e-7    &2.99    &1.23e-6    &2.99    &2.64e-5    &1.99\\
\hline
&\text{$p=3,$~$(\rho,\tau)=(1e+4,0)$}&&&&&&&\\
\hline
8   &5.63e-4    &2.07    &7.87e-6    &2.88    &4.71e-5    &2.87    &1.97e-4    &1.83\\
16  &1.37e-4    &2.04    &5.95e-7    &3.73    &3.57e-6    &3.72    &2.99e-5    &2.72\\
32  &3.35e-5    &2.03    &3.82e-8    &3.96    &2.29e-7    &3.96    &3.83e-6    &2.97\\
64  &8.26e-6    &2.02    &2.41e-9    &3.99    &1.44e-8    &3.99    &4.81e-7    &2.99\\
\hline
&\text{$p=5,$~$(\rho,\tau)=(1e+11,0)$}&&&&&&&\\
\hline
8   &5.07e-4    &2.25    &2.16e-6    &2.82    &9.44e-6    &2.65    &5.02e-5    &1.95\\
16  &1.23e-4    &2.04    &3.84e-8    &5.81    &1.76e-7    &5.75    &1.81e-6    &4.79\\
32  &3.02e-5    &2.03    &6.04e-10   &5.99    &2.76e-10   &5.99    &5.68e-8    &4.99\\
64  &7.46e-6    &2.02    &9.45e-12   &6.00    &4.32e-12   &6.00    &1.77e-9    &5.00\\
\hline
\end{tabular}
\end{table}

\begin{table}[htbp]
\tiny
\caption{Numerical errors and convergence rates for $L^p$-PDWG method with $k=2$ and $j=k-1$.}\label{NE6-2}
\begin{tabular}{p{0.4cm}p{1.34cm}p{0.80cm}p{1.43cm}p{0.8cm}p{1.43cm}p{0.8cm}p{1.34cm}p{0.8cm}}
\hline
$1/h$&$\|e_h\|_{0,q}$&Rate&$\|\epsilon_0\|_{0,p}$&Rate&$\|\epsilon_b\|_{0,p}$&Rate&$\|\epsilon_0\|_{1,p}$&Rate\\
\hline
&\text{$p=1.2,$~$(\rho,\tau)=(1,1)$}&&&&&&&\\
\hline
8   &1.39e-3    &1.99    &1.02e-7    &2.09    &4.89e-6    &2.06    &2.93e-6    &1.01\\
16  &3.71e-4    &1.90    &2.29e-8    &2.16    &1.09e-6    &2.16    &1.29e-6    &1.18\\
32  &1.07e-4    &1.80    &5.05e-9    &2.18    &2.41e-7    &2.18    &5.66e-7    &1.19\\
64  &3.23e-5    &1.73    &1.11e-9    &2.19    &5.28e-8    &2.19    &2.47e-7    &1.20\\
\hline
&\text{$p=1.6,$~$(\rho,\tau)=(1,1)$}&&&&&&&\\
\hline
8   &9.12e-4    &2.00    &2.07e-6    &2.64    &4.59e-5    &2.65    &6.04e-5    &1.56\\
16  &2.26e-4    &2.01    &3.43e-7    &2.59    &7.61e-6    &2.59    &2.00e-5    &1.59\\
32  &5.69e-5    &1.99    &5.71e-8    &2.58    &1.27e-6    &2.59    &6.66e-6    &1.59\\
64  &1.45e-5    &1.98    &9.49e-9    &2.59    &2.10e-7    &2.59    &2.21e-6    &1.59\\
\hline
&\text{$p=2,$~$(\rho,\tau)=(1,1)$}&&&&&&&\\
\hline
8   &7.84e-4    &2.02    &4.04e-5    &2.88    &5.60e-4    &2.89    &1.19e-3    &1.76\\
16  &1.92e-4    &2.03    &5.16e-6    &2.97    &7.19e-5    &2.96    &3.08e-4    &1.95\\
32  &4.77e-5    &2.01    &6.52e-7    &2.98    &9.10e-6    &2.98    &7.77e-5    &1.98\\
64  &1.19e-5    &2.00    &8.19e-8    &2.99    &1.14e-6    &2.99    &1.95e-5    &2.00\\
\hline
&\text{$p=3,$~$(\rho,\tau)=(1e+4,1e+3)$}&&&&&&&\\
\hline
8   &5.80e-4    &2.12    &7.13e-6    &2.86    &4.65e-5    &2.89    &1.85e-4    &1.81\\
16  &1.37e-4    &2.08    &5.74e-7    &3.63    &3.55e-6    &3.71    &2.92e-5    &2.66\\
32  &3.35e-5    &2.03    &3.77e-8    &3.93    &2.29e-7    &3.96    &3.80e-6    &2.94\\
64  &8.26e-6    &2.02    &2.39e-9    &3.98    &1.44e-8    &3.99    &4.79e-7    &2.99\\
\hline
&\text{$p=5,$~$(\rho,\tau)=(1e+11,1e+10)$}&&&&&&&\\
\hline
8   &5.07e-4    &2.36    &2.15e-6    &2.37    &9.44e-6    &2.59    &5.01e-5    &1.52\\
16  &1.23e-4    &2.04    &3.84e-8    &5.81    &1.76e-7    &5.75    &1.81e-6    &4.79\\
32  &3.02e-5    &2.03    &6.04e-10   &5.99    &2.76e-9    &5.99    &5.68e-8    &4.99\\
64  &7.46e-6    &2.02    &9.45e-12   &6.00    &4.32e-11   &6.00    &1.77e-9    &5.00\\
\hline
\end{tabular}
\end{table}

Table \ref{NE7} contains some numerical results for the model equation \eqref{model} with the exact solution $u=\sin(x)\cos(y)$ on the L-shaped domain $\Omega_2$  for $k=2$ and $j=k-1$ with different values of $p$. The convection vector is piece-wisely defined as follows: $\bbeta=[1,-1]$ when $x+y<1$ and $\bbeta=[-1,1]$ elsewhere. The reaction coefficient is $c=1$. These numerical results suggest that the optimal order of convergence can be seen for numerical error $e_h$ in $L^q$ norm. For the approximation of dual variable $\lambda_h$, the convergence rates are of order $r=p+1-m$ and $r=p+1$ for $\|\epsilon_0\|_{m,p}$ for $m=0,1$ and $\|\epsilon_b\|_{0,p}$, respectively.

\begin{table}[htbp]
\tiny
\caption{Numerical errors and convergence rates for $L^p$-PDWG method with $k=2$ and $j=k-1$.}\label{NE7}
\begin{tabular}{p{0.4cm}p{1.34cm}p{0.80cm}p{1.43cm}p{0.8cm}p{1.43cm}p{0.8cm}p{1.43cm}p{0.8cm}}
\hline
$1/h$&$\|e_h\|_{0,q}$&Rate&$\|\epsilon_0\|_{0,p}$&Rate&$\|\epsilon_b\|_{0,p}$&Rate&$\|\epsilon_0\|_{1,p}$&Rate\\
\hline
&\text{$p=1.2,$~$(\rho,\tau)=(1,0)$}&&&&&&&\\
\hline
8   &6.79e-4    &2.07    &5.84e-8    &2.04    &9.52e-7    &1.88    &2.39e-6    &1.05\\
16  &1.61e-4    &2.07    &1.32e-8    &2.14    &2.25e-7    &2.08    &1.08e-6    &1.14\\
32  &3.88e-5    &2.05    &2.92e-9    &2.18    &5.08e-8    &2.15    &4.78e-7    &1.18\\
64  &9.48e-6    &2.03    &6.41e-10   &2.19    &1.13e-8    &2.18    &2.10e-7    &1.19\\
\hline
&\text{$p=1.6,$~$(\rho,\tau)=(1,0)$}&&&&&&&\\
\hline
8   &3.57e-4    &2.04    &1.31e-6    &2.55    &1.16e-5    &2.41    &4.25e-5    &1.56\\
16  &8.75e-5    &2.03    &2.21e-7    &2.56    &2.03e-6    &2.51    &1.44e-5    &1.57\\
32  &2.16e-5    &2.02    &3.71e-8    &2.58    &3.44e-7    &2.56    &4.80e-6    &1.58\\
64  &5.37e-6    &2.01    &6.16e-9    &2.59    &5.76e-8    &2.58    &1.60e-6    &1.59\\
\hline
&\text{$p=2,$~$(\rho,\tau)=(1,0)$}&&&&&&&\\
\hline
8   &2.87e-4    &2.02    &2.68e-5    &2.85    &1.62e-4    &2.72    &7.43e-4    &1.86\\
16  &7.11e-5    &2.01    &3.48e-6    &2.94    &2.18e-5    &2.90    &1.93e-4    &1.95\\
32  &1.77e-5    &2.01    &4.42e-7    &2.98    &2.80e-6    &2.96    &4.90e-5    &1.98\\
64  &4.41e-6    &2.00    &5.57e-8    &2.99    &3.55e-7    &2.98    &1.23e-5    &1.99\\
\hline
&\text{$p=3,$~$(\rho,\tau)=(1e+4,0)$}&&&&&&&\\
\hline
8   &2.33e-4    &1.97    &4.18e-6    &3.12    &1.51e-5    &3.02    &9.09e-5    &2.14\\
16  &5.85e-5    &2.00    &2.97e-7    &3.81    &1.10e-6    &3.77    &1.29e-5    &2.81\\
32  &1.46e-5    &2.00    &1.90e-8    &3.97    &7.13e-8    &3.95    &1.65e-6    &2.97\\
64  &3.65e-6    &2.00    &1.20e-9    &3.99    &4.51e-9    &3.98    &2.08e-7    &2.99\\
\hline
&\text{$p=5,$~$(\rho,\tau)=(1e+12,0)$}&&&&&&&\\
\hline
8   &2.07e-4    &1.97    &1.05e-7    &5.50    &2.42e-7    &5.38    &1.78e-6    &4.52\\
16  &5.20e-5    &2.00    &1.71e-9    &5.95    &4.03e-9    &5.91    &5.78e-8    &4.95\\
32  &1.30e-5    &2.00    &2.70e-11   &5.98    &6.45e-11   &5.97    &1.83e-9    &4.98\\
64  &3.25e-6    &2.00    &4.25e-13   &5.99    &1.02e-12   &5.99    &5.75e-11   &4.99\\
\hline
\end{tabular}
\end{table}

Table \ref{NE8} presents some numerical rates of convergence for the numerical approximation $(u_h;\lambda_h)$ on the unit square domain $\Omega_1$ for $k=2$ and $j=k-1$  with different values of $p$. The exact solution is chosen as follows
  \begin{equation*}
\begin{split}
& u=\begin{cases}\cos(y-1/2) + \sin(x + y),& \mbox{if}~~y<\frac{1}{2},\\
  1+ \sin(x + y),& \mbox{otherwise,}\end{cases}
\end{split}
\end{equation*}
which implies $u\in H^{1+\epsilon}$ for $\epsilon>0$. The convection vector is piece-wisely defined such that $\bbeta=[x-2,\frac{1}{2}-y]$ when $y<1/2$ and $\bbeta=[2-x,-y+\frac{1}{2}]$ otherwise. The reaction coefficient is $c=0$. As we can observe from these numerical results, the convergence rate for $\|e_h\|_{0,q}$ varies from ${\cal O}(h^{1.2})$ to ${\cal O}(h^{1.9})$ as $p$ increases from $p=1.2$ to $p=5$. This indicates the convergence rate for $\|e_h\|_{0,q}$ seems to be related to $p$. As for the dual variable $\lambda_h$, the numerical errors $\|\epsilon_0\|_{m,p}$ and $\|\epsilon_b\|_{0,p}$ converge at the rates of ${\cal O}(h^{p+1-m})$ for $m=0,1$ and ${\cal O}(h^{p+1})$, respectively.

\begin{table}[htbp]
\tiny
\caption{Numerical errors and convergence rates for $L^p$-PDWG method with $k=2$ and $j=k-1$.}\label{NE8}
\begin{tabular}{p{0.4cm}p{1.34cm}p{0.80cm}p{1.43cm}p{0.8cm}p{1.43cm}p{0.8cm}p{1.43cm}p{0.8cm}}
\hline
$1/h$&$\|e_h\|_{0,q}$&Rate&$\|\epsilon_0\|_{0,p}$&Rate&$\|\epsilon_b\|_{0,p}$&Rate&$\|\epsilon_0\|_{1,p}$&Rate\\
\hline
&\text{$p=1.2,$~$(\rho,\tau)=(1,0)$}&&&&&&&\\
\hline
16  &3.49e-4    &1.59    &6.48e-8    &2.18    &1.36e-6    &2.16    &6.47e-6    &1.18\\
32  &1.42e-4    &1.30    &1.42e-8    &2.19    &3.01e-7    &2.17    &2.84e-6    &1.19\\
64  &6.15e-5    &1.20    &3.11e-9    &2.19    &6.62e-8    &2.18    &1.24e-6    &1.20\\
128 &2.72e-5    &1.18    &6.78e-10   &2.20    &1.45e-8    &2.19    &5.40e-7    &1.20\\
\hline
&\text{$p=1.6,$~$(\rho,\tau)=(1,0)$}&&&&&&&\\
\hline
16  &1.96e-4    &1.83    &9.99e-7    &2.63    &1.04e-5    &2.62    &8.52e-5    &1.63\\
32  &6.11e-5    &1.69    &1.65e-7    &2.60    &1.73e-6    &2.59    &2.82e-5    &1.60\\
64  &2.11e-5    &1.53    &2.74e-8    &2.59    &2.87e-7    &2.59    &9.31e-6    &1.60\\
128 &7.82e-6    &1.44    &4.52e-9    &2.60    &4.76e-8    &2.59    &3.08e-6    &1.60\\
\hline
&\text{$p=2,$~$(\rho,\tau)=(1,0)$}&&&&&&&\\
\hline
16  &1.66e-4    &1.89    &1.50e-5    &2.98    &1.05e-4    &2.97    &1.16e-3    &1.94\\
32  &4.71e-5    &1.82    &1.89e-6    &2.99    &1.33e-5    &2.98    &2.91e-4    &1.97\\
64  &1.42e-5    &1.73    &2.37e-7    &2.99    &1.67e-6    &2.99    &7.31e-5    &1.99\\
128 &4.51e-6    &1.65    &2.98e-8    &3.00    &2.09e-7    &3.00    &1.83e-5    &2.00\\
\hline
&\text{$p=3,$~$(\rho,\tau)=(1e+4,0)$}&&&&&&&\\
\hline
16  &1.44e-4    &1.94    &1.26e-6    &3.69    &5.27e-6    &3.67    &8.21e-5    &2.68\\
32  &3.85e-5    &1.90    &8.06e-8    &3.96    &3.38e-7    &3.96    &1.05e-5    &2.96\\
64  &1.05e-5    &1.87    &5.06e-9    &3.99    &2.12e-8    &3.99    &1.32e-6    &2.99\\
128 &2.92e-6    &1.85    &3.17e-10   &4.00    &1.33e-9    &4.00    &1.66e-7    &3.00\\
\hline
&\text{$p=5,$~$(\rho,\tau)=(1e+13,0)$}&&&&&&&\\
\hline
16  &1.43e-4    &1.97    &8.10e-10   &5.99    &2.17e-9    &5.99    &4.34e-8    &4.98\\
32  &3.73e-5    &1.94    &1.27e-11   &5.99    &3.40e-11   &6.00    &1.36e-9    &4.99\\
64  &9.81e-6    &1.93    &1.99e-13   &6.00    &5.32e-13   &6.00    &4.27e-11   &5.00\\
128 &2.59e-6    &1.92    &3.00e-15   &6.00    &8.00e-15   &6.00    &1.34e-12   &5.00\\
\hline
\end{tabular}
\end{table}

\subsection{Plots for numerical solutions}
Figure \ref{NE9} demonstrates the surface plots for the numerical solutions in $\Omega_1$ with different values of $p$. The test problem has the following configuration: (1) the convection vector is piece-wisely defined in the sense that $\bbeta=[1,-1]$ when $x+y<1$ and $\bbeta=[-2,2]$ otherwise; (2) the reaction coefficient is $c=0$; (3) the inflow boundary data is $g=1$ on edge $\{0\}\times(0,1)$ and $g=-1$ on edge $\{1\}\times(0,1)$; (4) the load function is $f=0$; (5) the stabilization parameter is $\tau=0$; (6) $k=2$ and $j=k-1$. The left ones in Figure \ref{NE9} are the surface plots for the primal variable $u_h$, and the right ones are the surface plots for the dual variable $\lambda_0$. The exact solution of primal variable is $u=1$ for $y<1-x$ and $u=-1$ for $y>1-x$. Note that the exact solution of dual variable is $0$. The plots indicate that the numerical solution $u_h$ is in perfect consistency with the exact solution $u$, and the dual variable $\lambda_0$ is zero everywhere in $\Omega_1$ upto machine accuracy.

\begin{figure}[!htbp]
\centering
\begin{tabular}{cccc}
\resizebox{2.4in}{1.6in}{\includegraphics{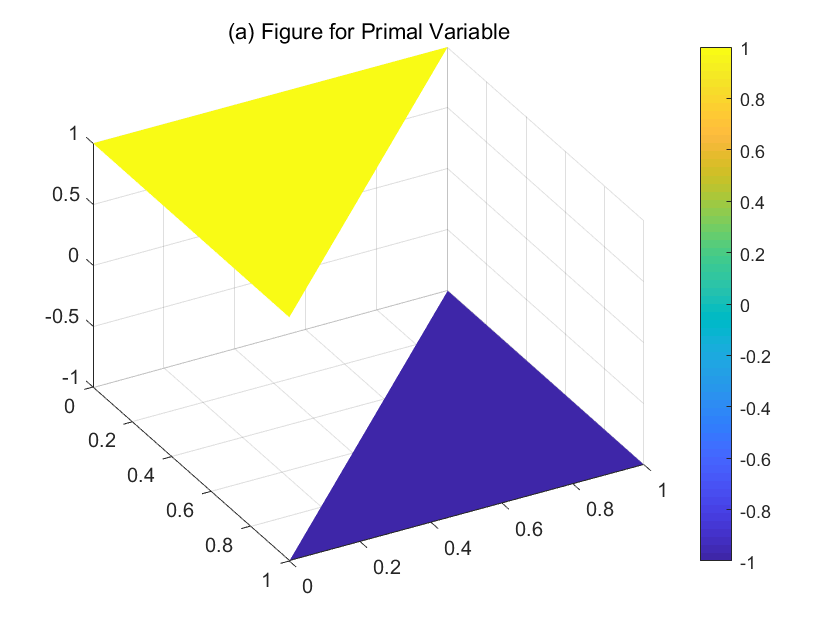}}
\resizebox{2.4in}{1.6in}{\includegraphics{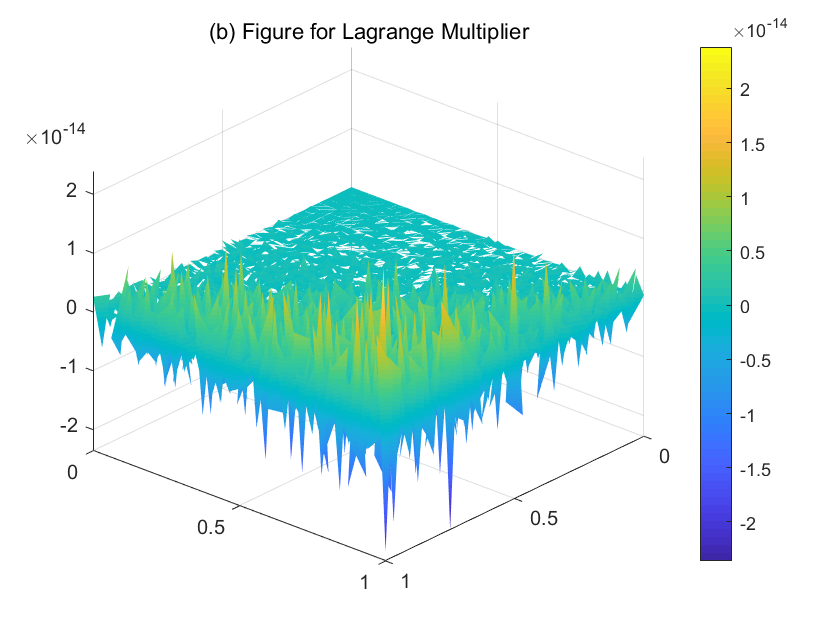}}\\
\resizebox{2.4in}{1.6in}{\includegraphics{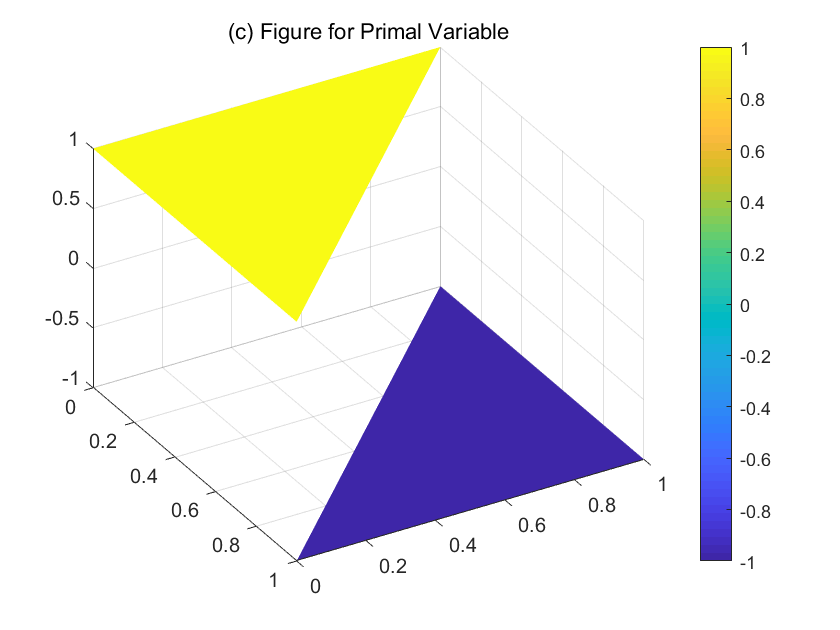}}
\resizebox{2.4in}{1.6in}{\includegraphics{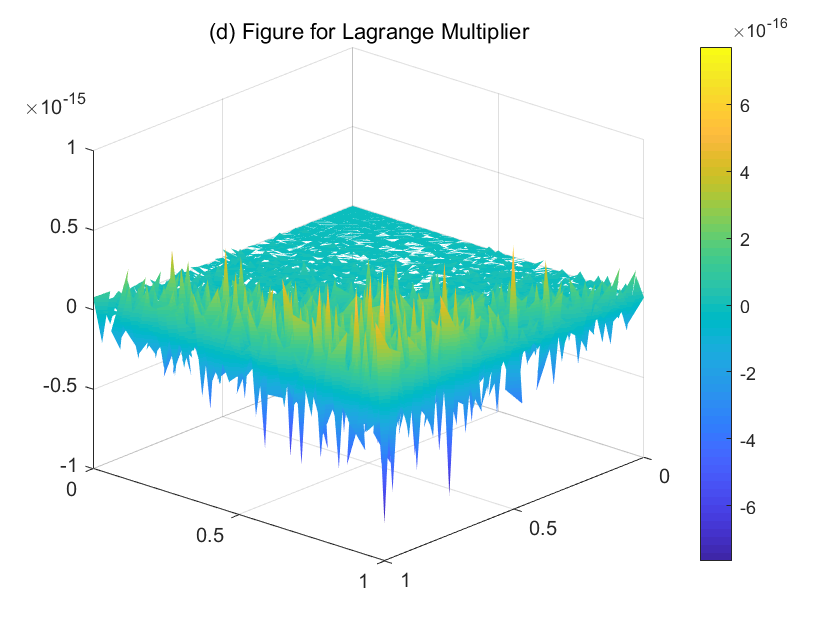}}
\end{tabular}
\caption{Surface plots for numerical solutions with $k=2$ and $j=k-1$ on $\Omega_1$, (a) the primal variable $u_h$ and (b) the Lagrange multiplier $\lambda_0$ for $\{p,\rho,\tau\}=\{1.2,1,0\}$, (c) the primal variable $u_h$ and (d) the Lagrange multiplier $\lambda_0$ for $\{p,\rho,\tau\}=\{5,1e+11,0\}$.}\label{NE9}
\end{figure}

Figures \ref{NE10}-\ref{NE11} present some contour plots for the primal variable $u_h$ and the dual variable $\lambda_0$ in $\Omega_1$ for the cases of $p=1.2$ and $p=5$ respectively. The following configuration is as follows: (1) the convection vector is piece-wisely defined such that $\bbeta=[-y,x]$ if $y<1-x$ and $\bbeta=[1-y,x-1]$ otherwise; (2) the reaction coefficient is $c=0$; (3) the inflow boundary data is given by $g =\sin(2x)$; (4) the stabilization parameter is $\tau=0$; (5) $k=2$ and $j=k-1$.  Note that the exact solution of model problem \eqref{model} is unknown.

\begin{figure}[!htbp]
\centering
\begin{tabular}{cccc}
\resizebox{2.25in}{1.7in}{\includegraphics{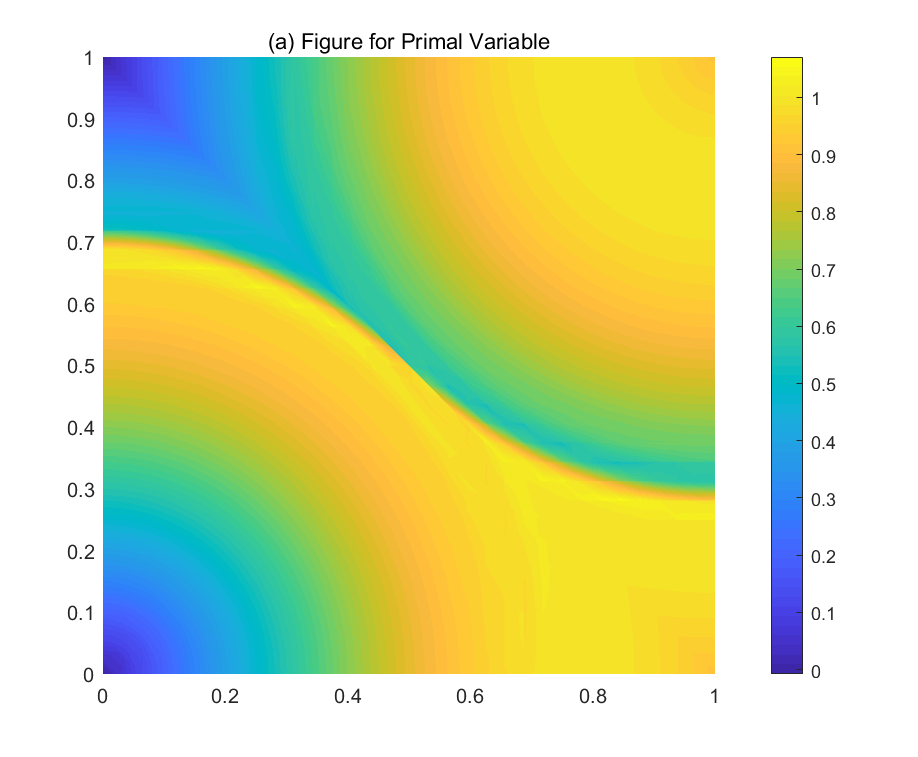}}
\resizebox{2.25in}{1.7in}{\includegraphics{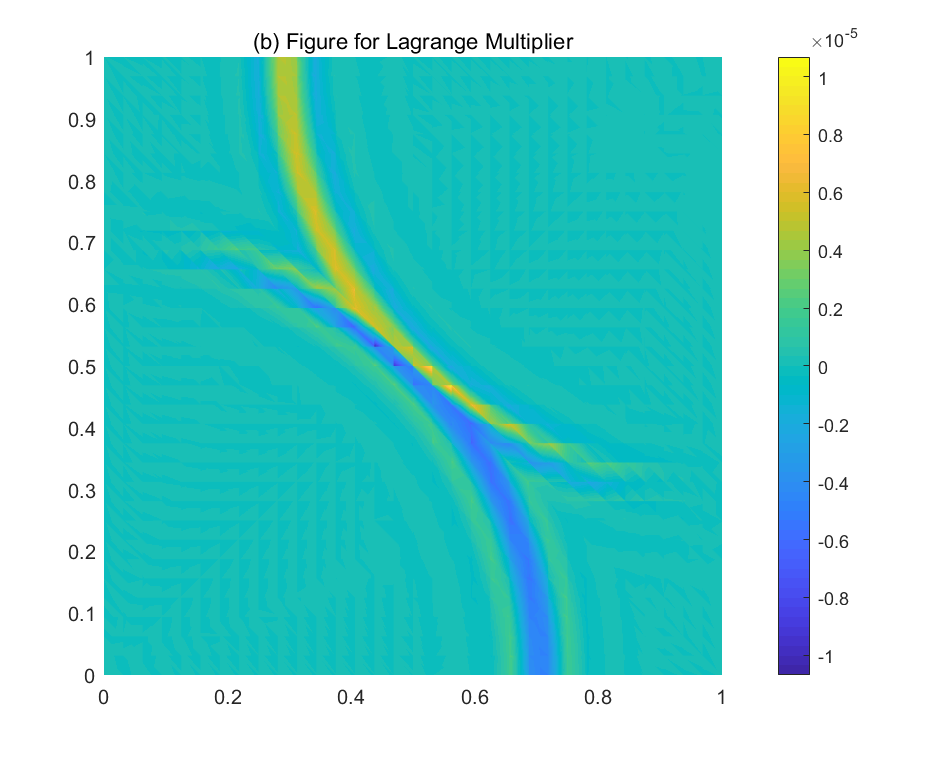}}\\
\resizebox{2.25in}{1.7in}{\includegraphics{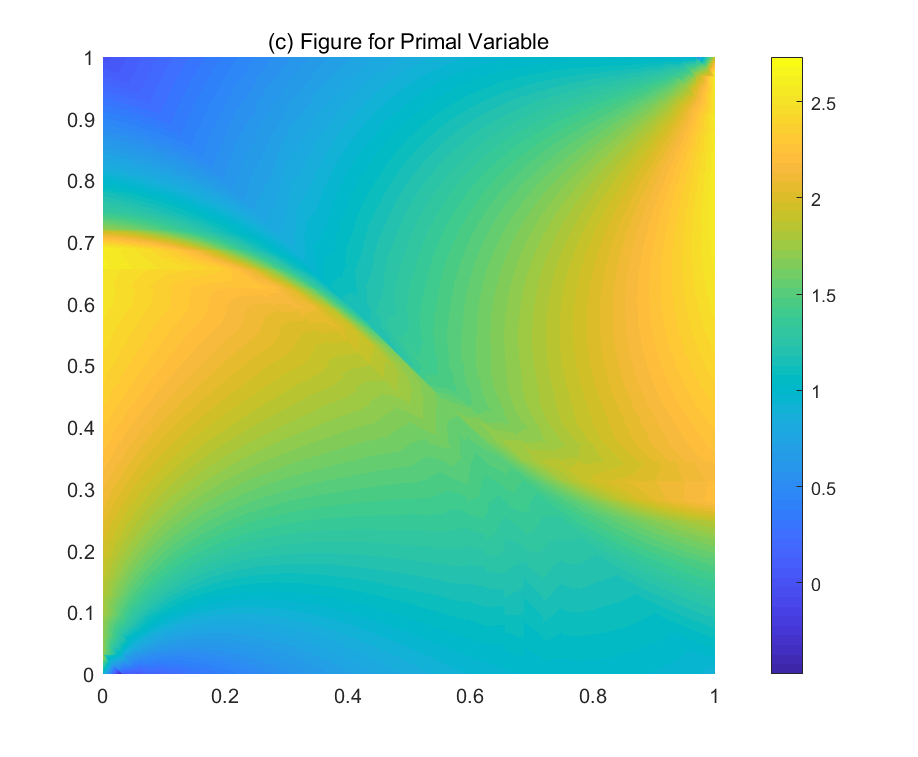}}
\resizebox{2.25in}{1.7in}{\includegraphics{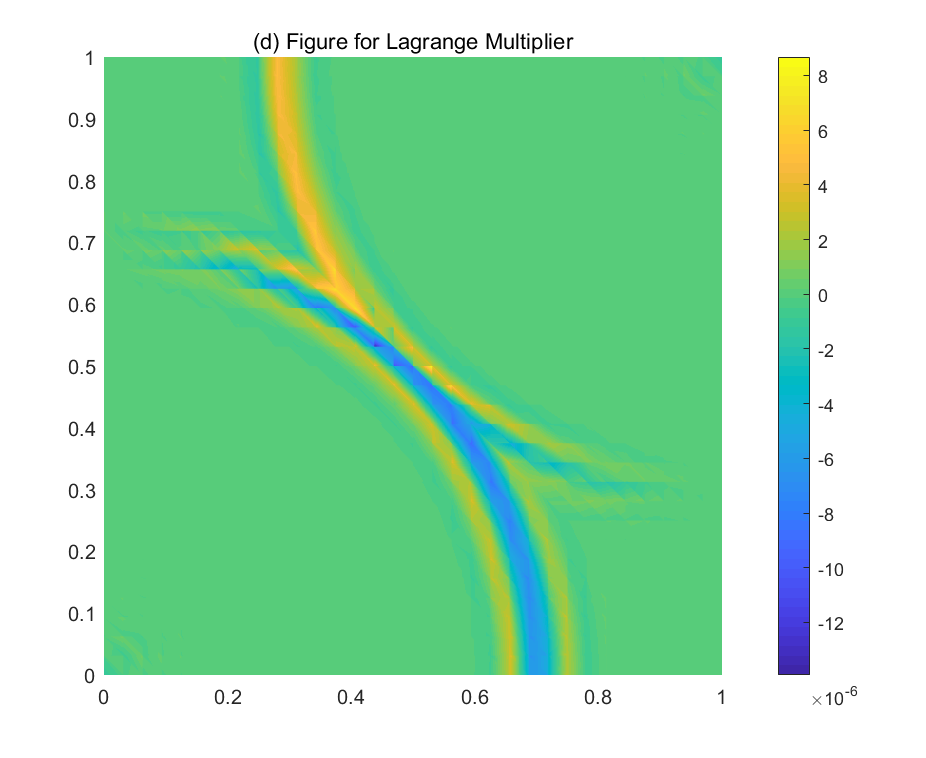}}
\end{tabular}
\caption{Contour plots for numerical solutions with $k=2$ and $j=k-1$ on $\Omega_1$, (a) the primal variable $u_h$ and (b) the Lagrange multiplier $\lambda_0$ for $f=0$, (c) the primal variable $u_h$ and (d) the Lagrange multiplier $\lambda_0$ for $f=1$, the coefficients $\{p,\rho,\tau\}=\{1.2,1,0\}$.}\label{NE10}
\end{figure}

\begin{figure}[!htbp]
\centering
\begin{tabular}{cccc}
\resizebox{2.25in}{1.7in}{\includegraphics{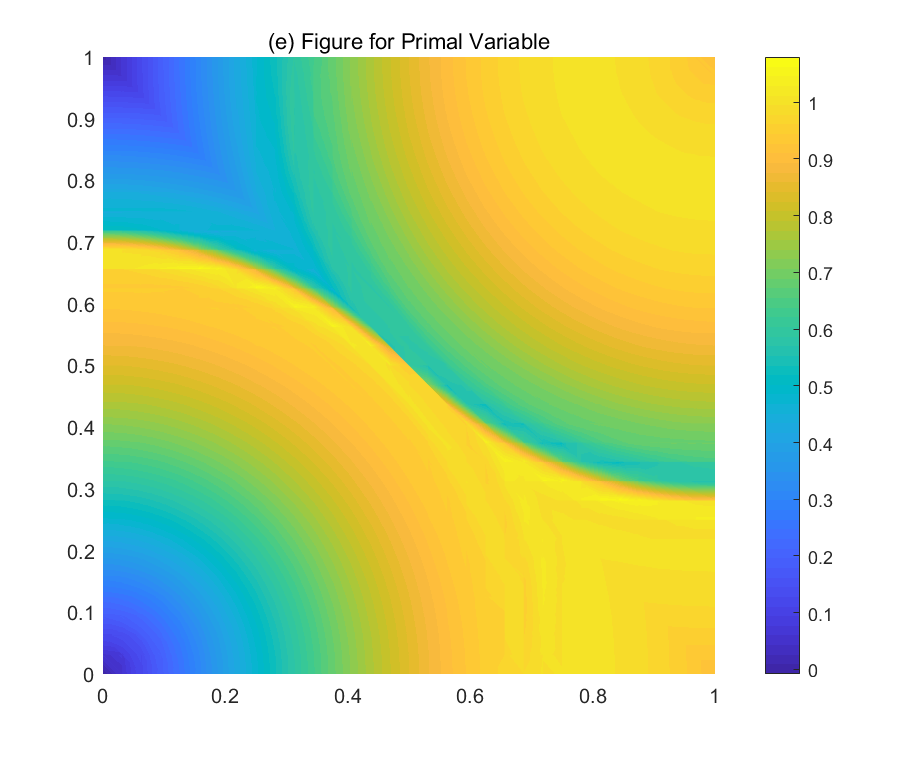}}
\resizebox{2.25in}{1.7in}{\includegraphics{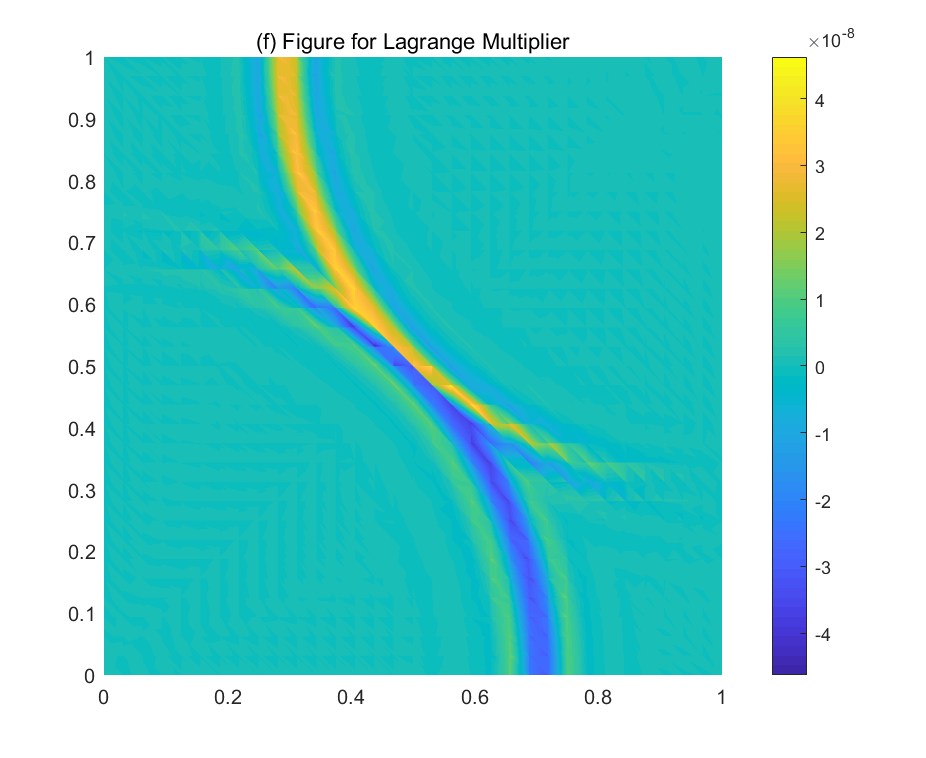}}\\
\resizebox{2.25in}{1.7in}{\includegraphics{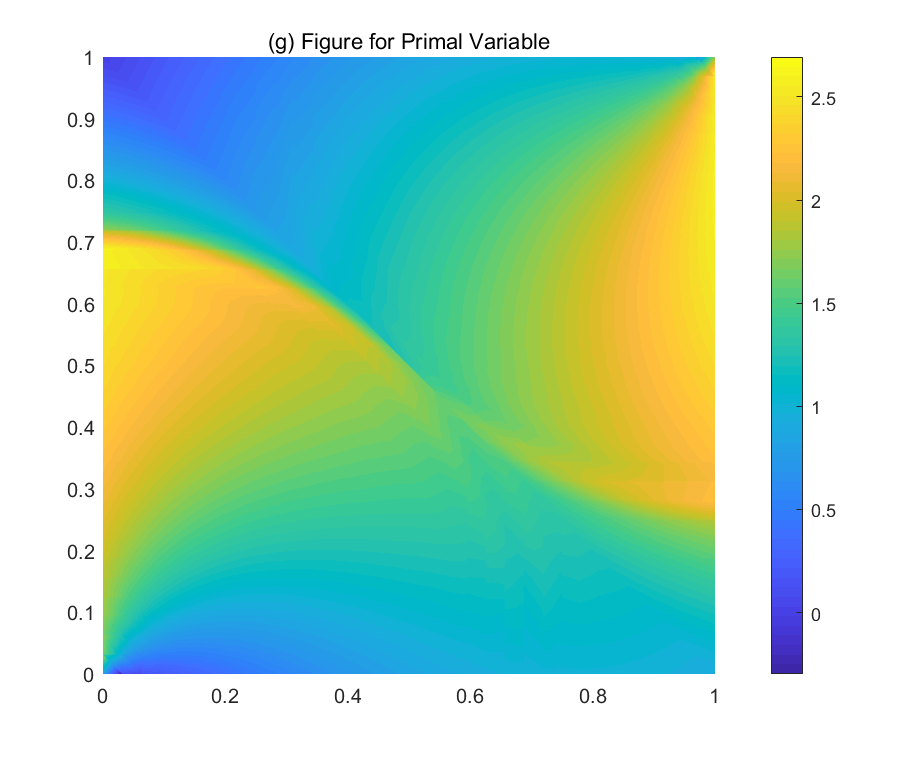}}
\resizebox{2.25in}{1.7in}{\includegraphics{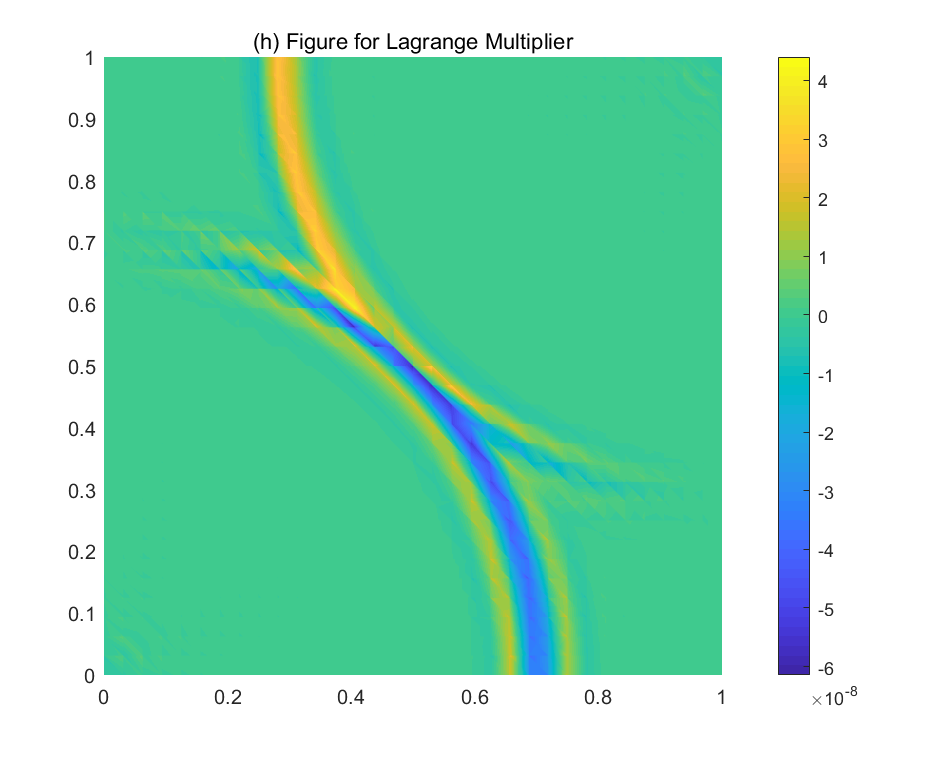}}
\end{tabular}
\caption{Contour plots for numerical solutions with $k=2$ and $j=k-1$ on $\Omega_1$, (e) the primal variable $u_h$ and (f) the Lagrange multiplier $\lambda_0$ for $f=0$, (g) the primal variable $u_h$ and (h) the Lagrange multiplier $\lambda_0$ for $f=1$, the coefficients $\{p,\rho,\tau\}=\{5,1e+12,0\}$.}\label{NE11}
\end{figure}

Figures \ref{NE12}-\ref{NE13} illustrate some contour plots for the primal variable $u_h$ and the dual variable $\lambda_0$ in $\Omega_1$ for $k=2$ and $j=k-1$ when $p=1.2$ and $p=5$ are employed. The convection vector is given by $\bbeta=[0.5-y,x-0.5]$, and the reaction coefficient is $c=1$. The inflow boundary data is $g =\cos(\pi x)\cos(\pi y)$. The left ones in Figures \ref{NE12}-\ref{NE13} show the contour plots of the primal variable $u_h$ corresponding to the load functions $f=0$ and $f=1$ respectively; while the right ones are the contour plots of the dual variable $\lambda_0$ for the load functions $f=0$ and $f=1$ respectively. Note that the exact solution of primal variable is unknown.

\begin{figure}[!htbp]
\centering
\begin{tabular}{cccc}
\resizebox{2.25in}{1.7in}{\includegraphics{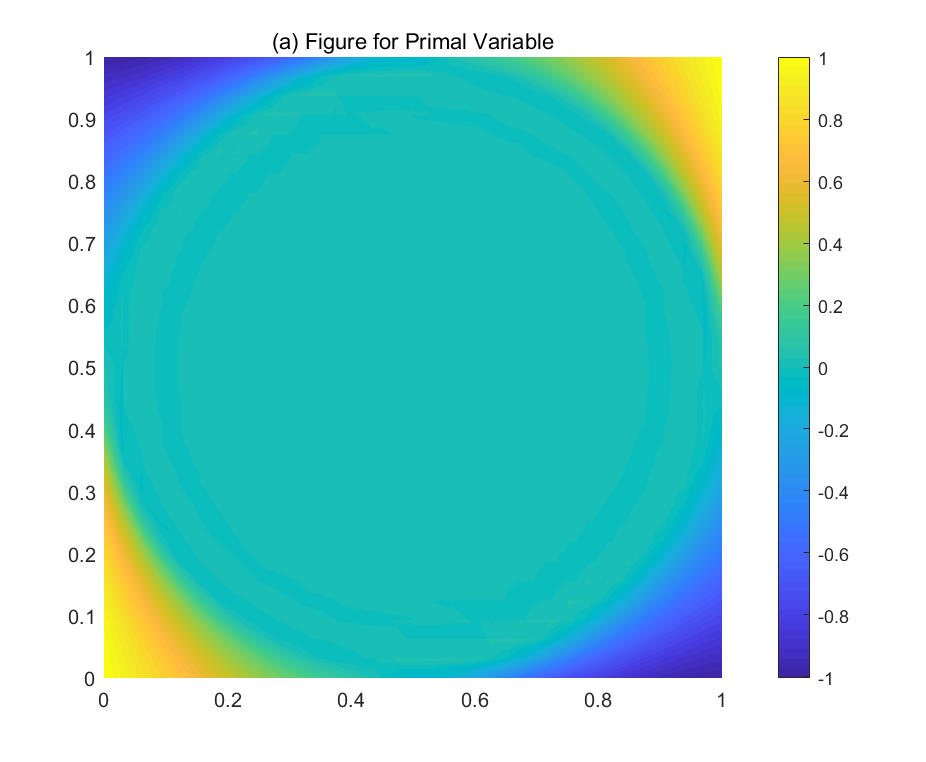}}
\resizebox{2.25in}{1.7in}{\includegraphics{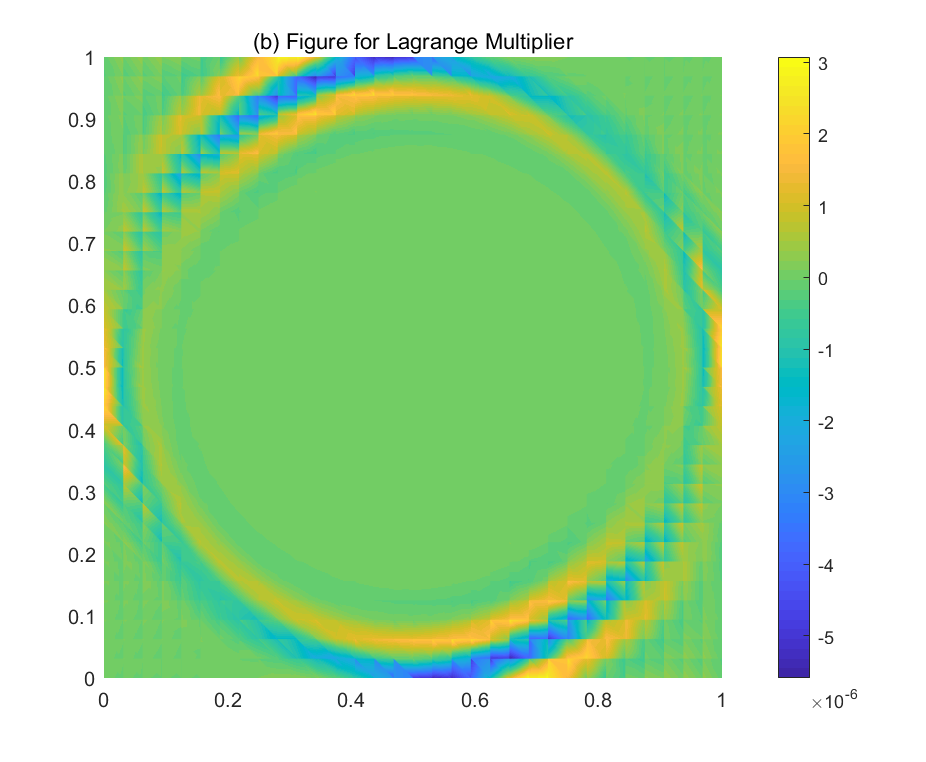}}\\
\resizebox{2.25in}{1.7in}{\includegraphics{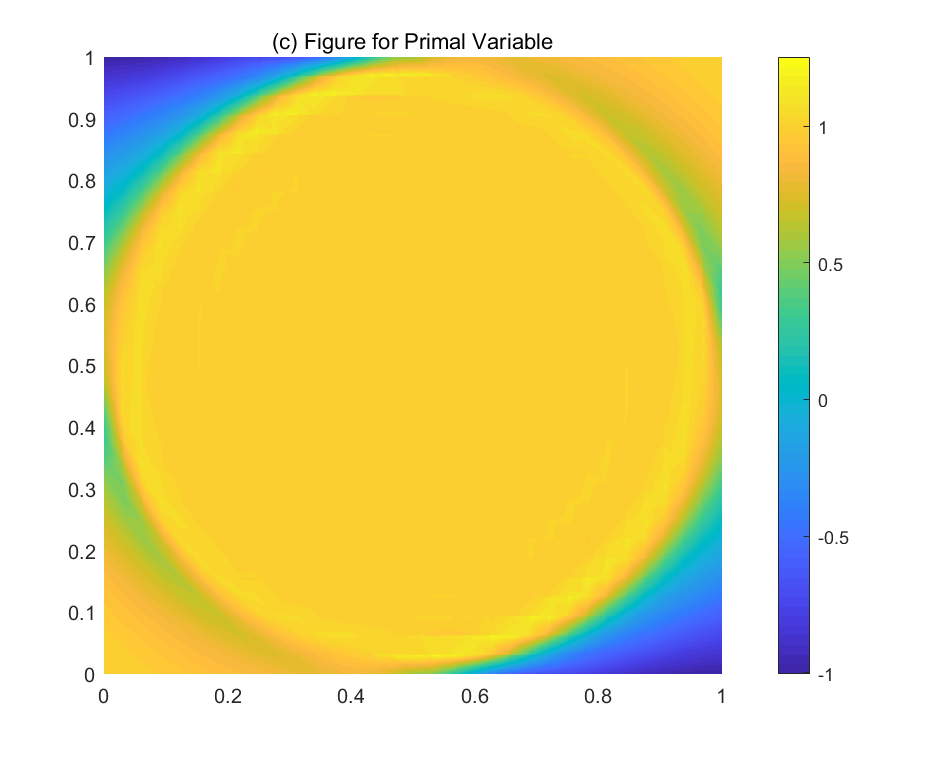}}
\resizebox{2.25in}{1.7in}{\includegraphics{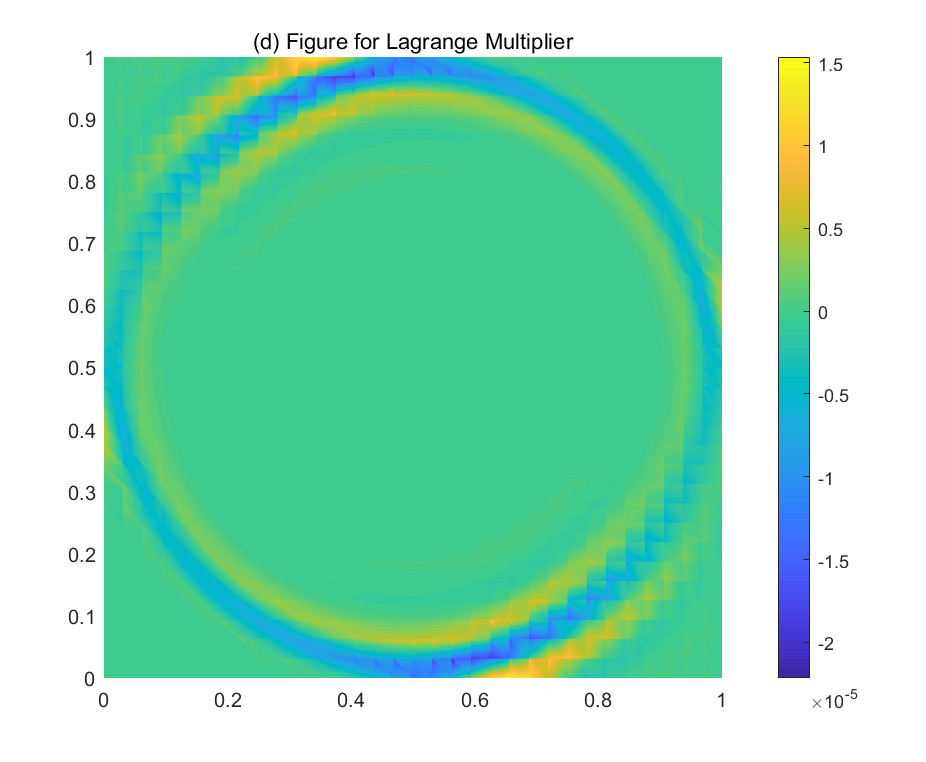}}
\end{tabular}
\caption{Contour plots for numerical solutions with $k=2$ and $j=k-1$ on $\Omega_1$, (a) the primal variable $u_h$ and (b) the Lagrange multiplier $\lambda_0$ for $f=0$, (c) the primal variable $u_h$ and (d) the Lagrange multiplier $\lambda_0$ for $f=1$, the coefficients $\{p,\rho,\tau\}=\{1.2,1,0\}$.}\label{NE12}
\end{figure}

\begin{figure}[!htbp]
\centering
\begin{tabular}{cccc}
\resizebox{2.25in}{1.7in}{\includegraphics{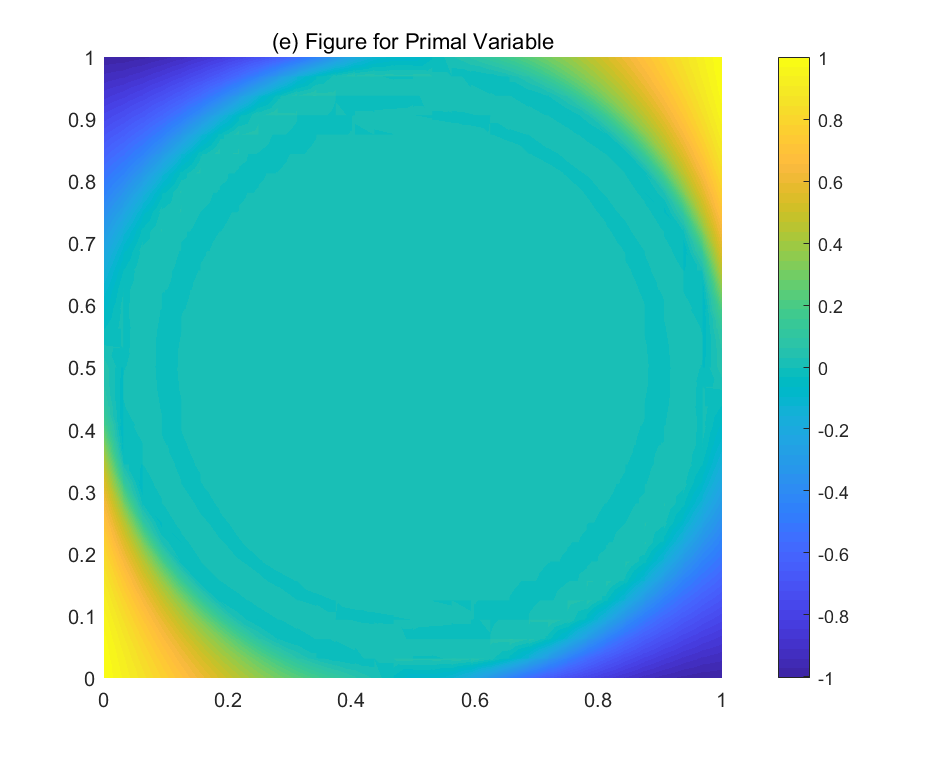}}
\resizebox{2.25in}{1.7in}{\includegraphics{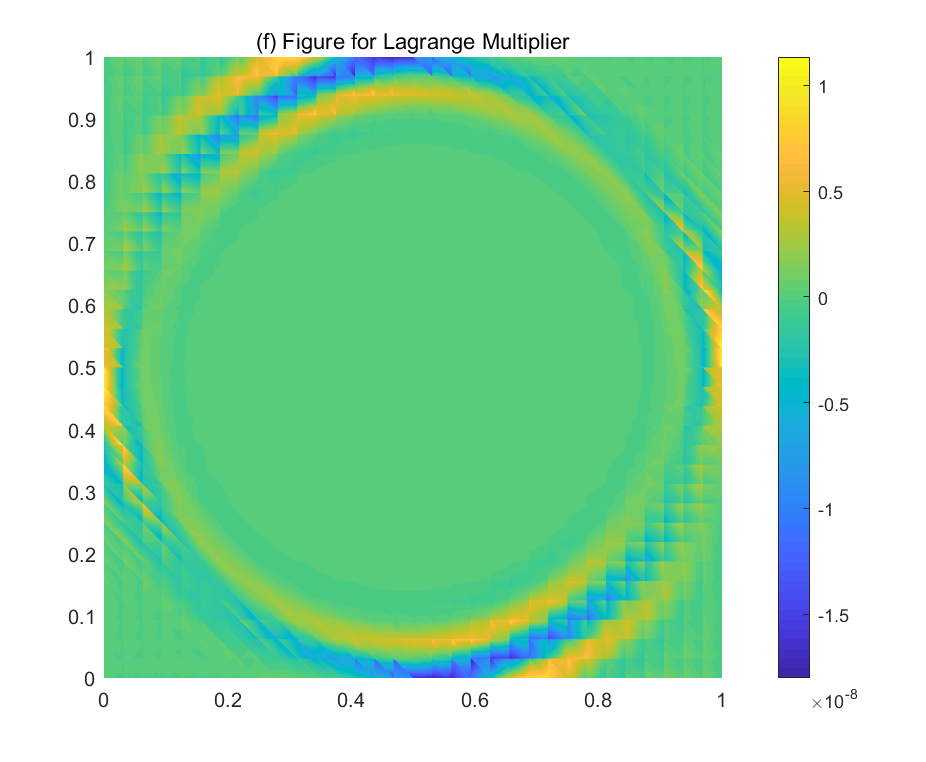}}\\
\resizebox{2.25in}{1.7in}{\includegraphics{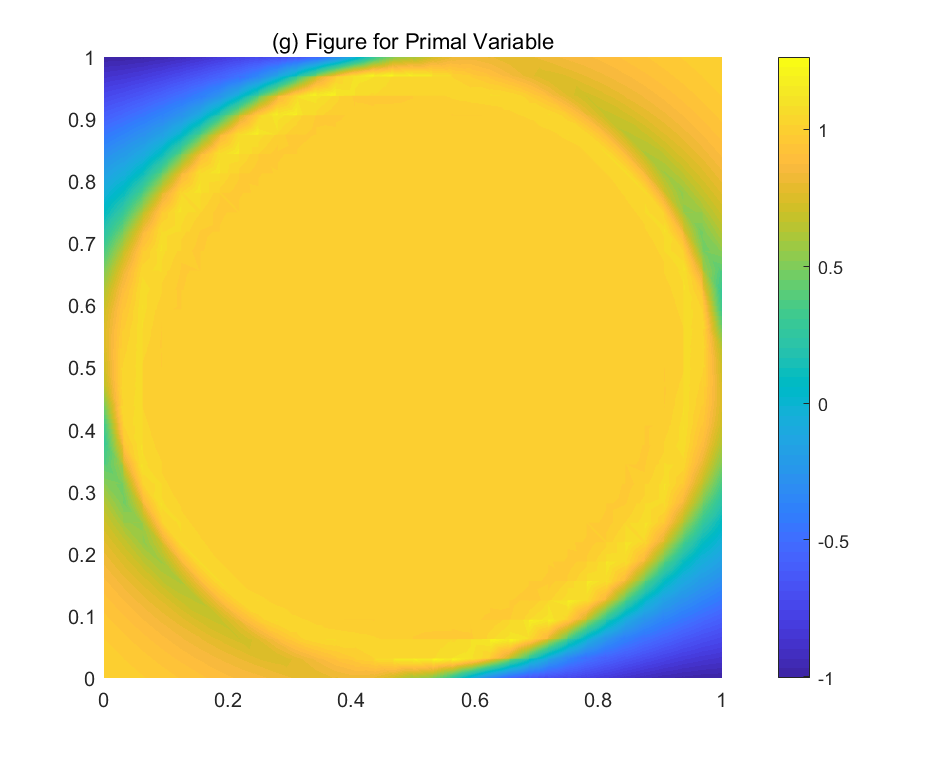}}
\resizebox{2.25in}{1.7in}{\includegraphics{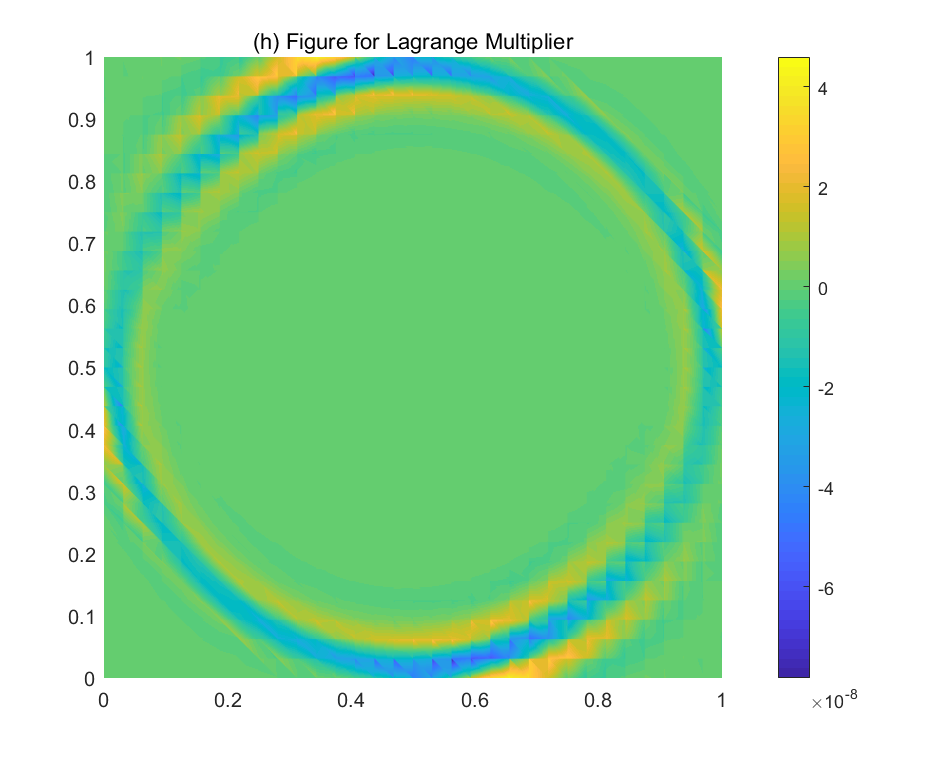}}
\end{tabular}
\caption{Contour plots for numerical solutions with $k=2$ and $j=k-1$ on $\Omega_1$, (e) the primal variable $u_h$ and (f) the Lagrange multiplier $\lambda_0$ for $f=0$, (g) the primal variable $u_h$ and (h) the Lagrange multiplier $\lambda_0$ for $f=1$, the coefficients $\{p,\rho,\tau\}=\{5,1e+12,0\}$.}\label{NE13}
\end{figure}

Figures \ref{NE14}-\ref{NE15} show the contour plots for the primal variable $u_h$ and the dual variable $\lambda_0$ in the domain $\Omega_2$ for the load function $f=0$ and $f=1$ respectively.  The configuration of this test problem is as follows: (1) the convection vector is piece-wisely defined such that $\bbeta=[y+1,-x-1]$ if $y<1-x$ and $\bbeta=[y-2,2-x]$ otherwise; (2) the reaction coefficient is $c=0$; (3) the inflow boundary data is $g =\cos(5y)$; (4) $k=2$ and $j=k-1$. We take $p=1.2$ and $p=5$.  No exact solution of primal variable is known.

\begin{figure}[!htbp]
\centering
\begin{tabular}{cccc}
\resizebox{2.25in}{1.7in}{\includegraphics{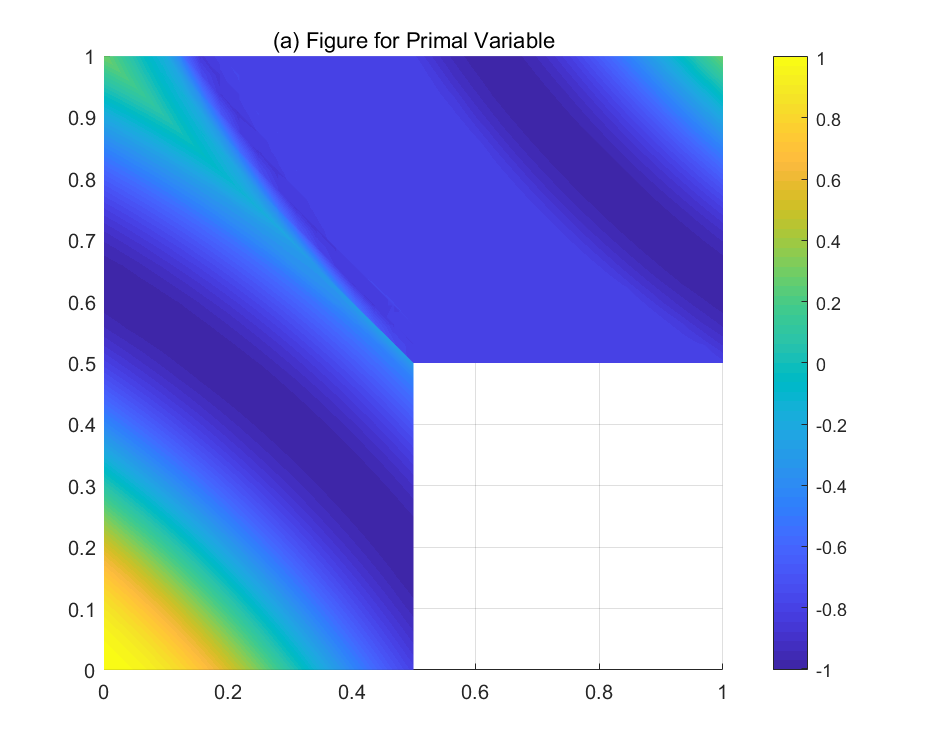}}
\resizebox{2.25in}{1.7in}{\includegraphics{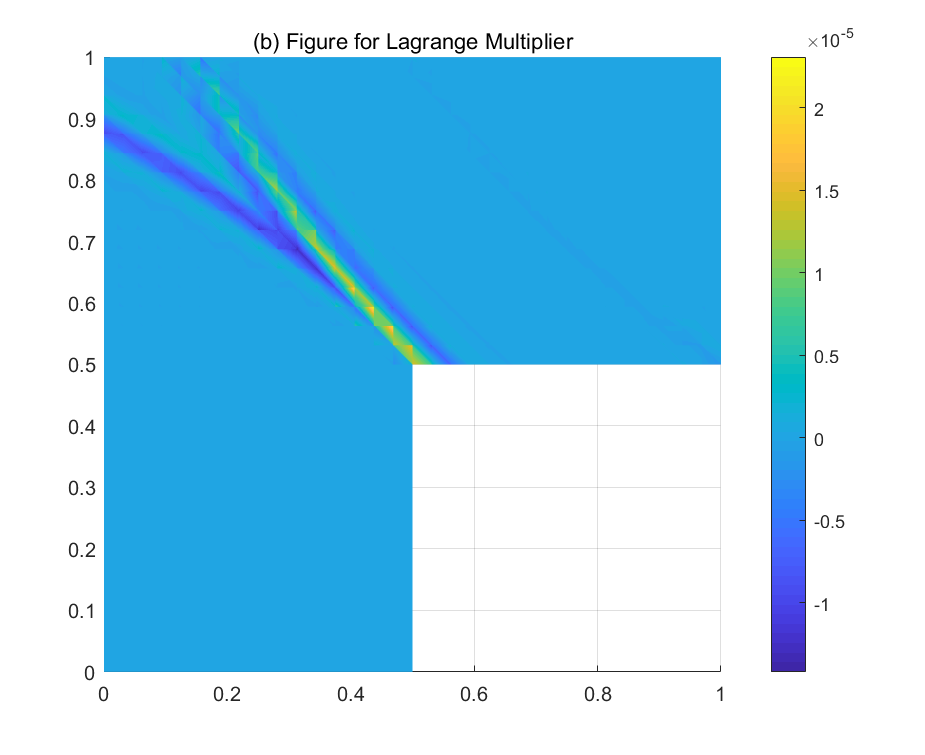}}\\
\resizebox{2.25in}{1.7in}{\includegraphics{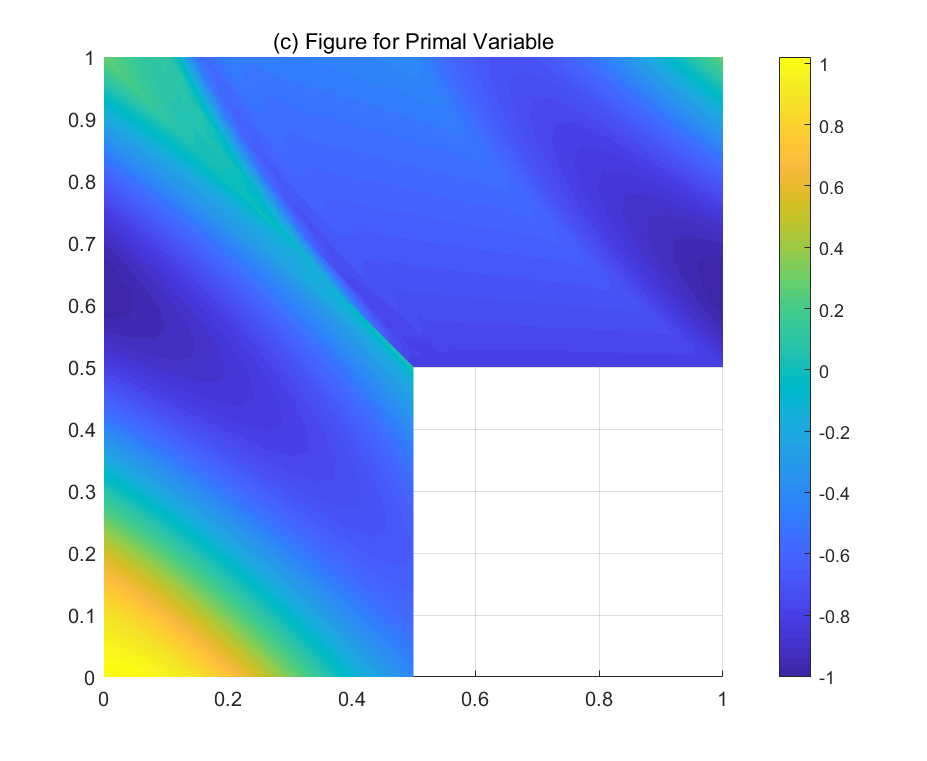}}
\resizebox{2.25in}{1.7in}{\includegraphics{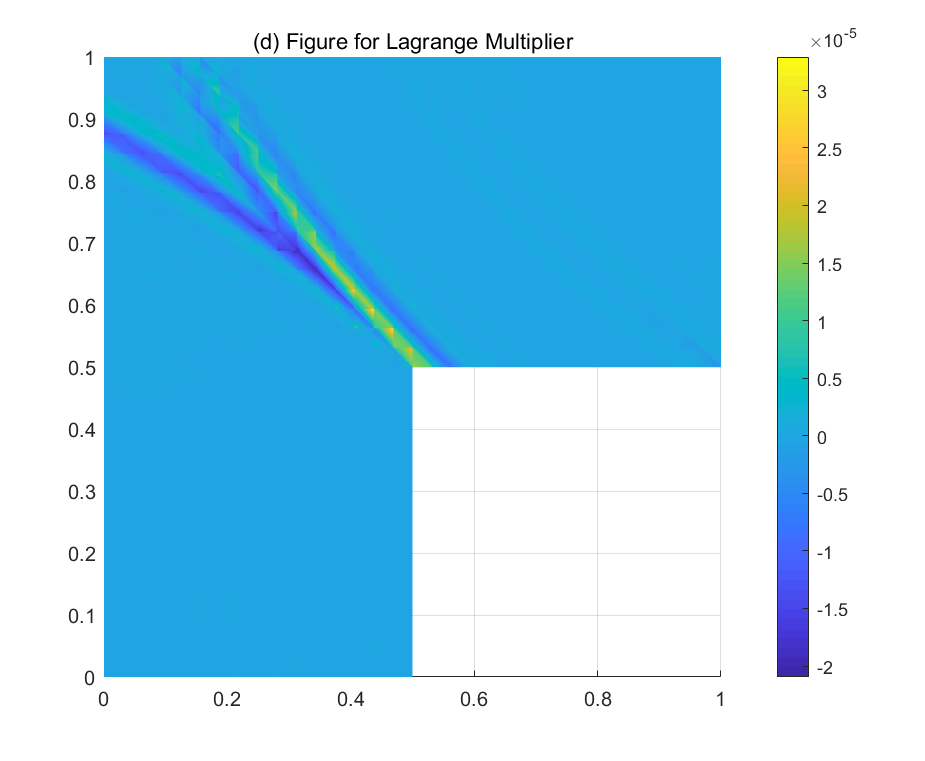}}
\end{tabular}
\caption{Contour plots for numerical solutions with $k=2$ and $j=k-1$ on $\Omega_2$, (a) the primal variable $u_h$ and (b) the Lagrange multiplier $\lambda_0$ for $f=0$, (c) the primal variable $u_h$ and (d) the Lagrange multiplier $\lambda_0$ for $f=1$, the coefficients $\{p,\rho,\tau\}=\{1.2,1,0\}$.}\label{NE14}
\end{figure}

\begin{figure}[!htbp]
\centering
\begin{tabular}{cccc}
\resizebox{2.25in}{1.7in}{\includegraphics{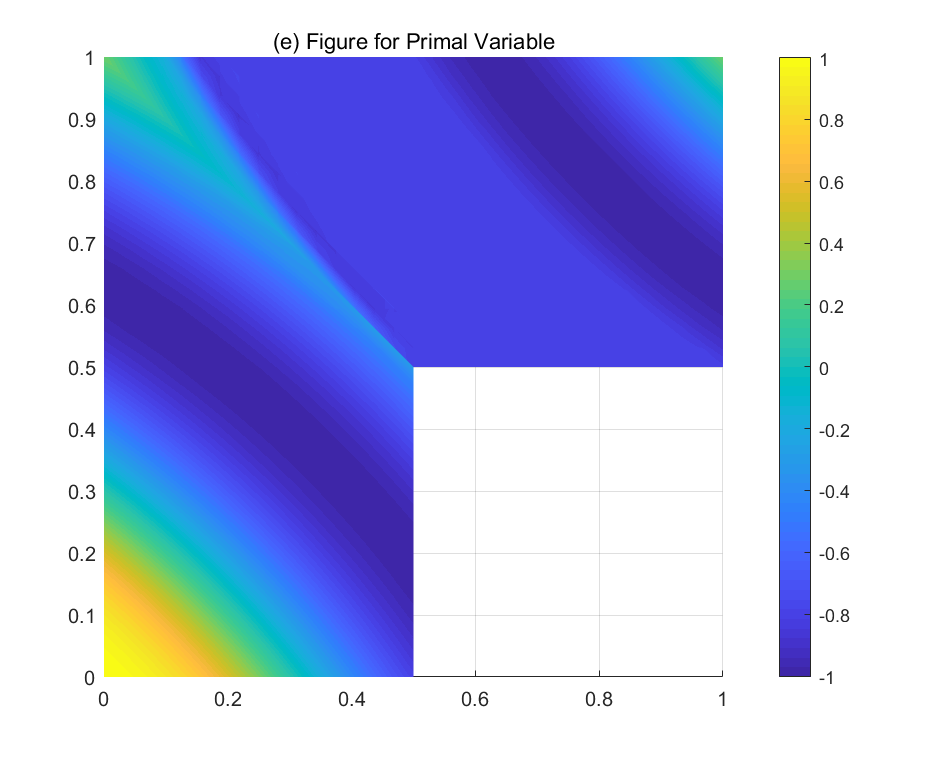}}
\resizebox{2.25in}{1.7in}{\includegraphics{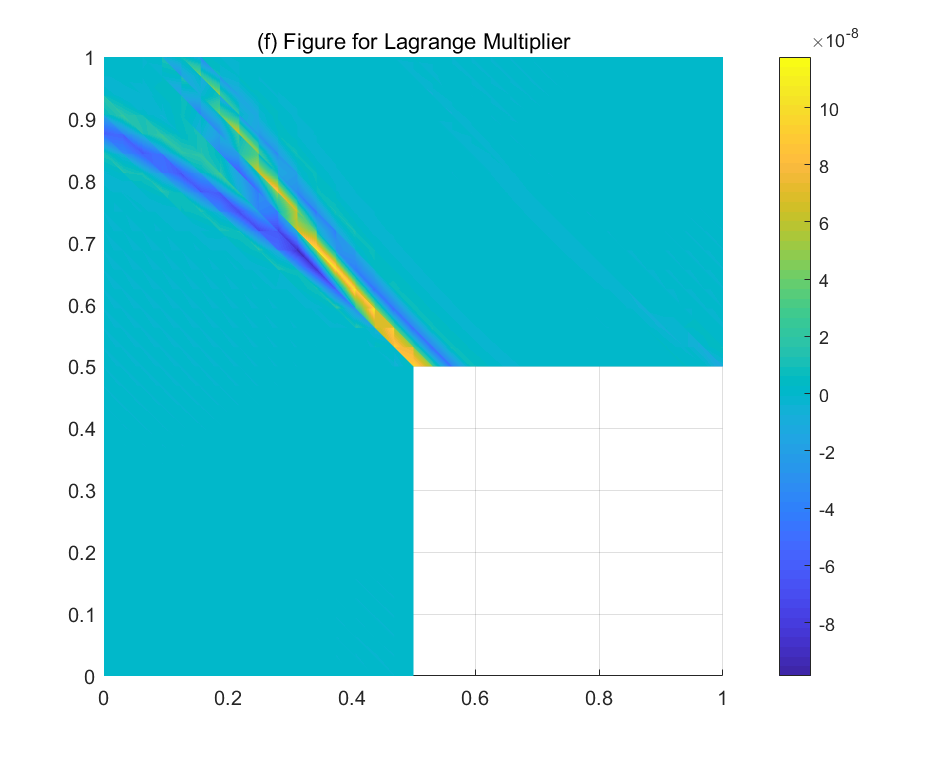}}\\
\resizebox{2.25in}{1.7in}{\includegraphics{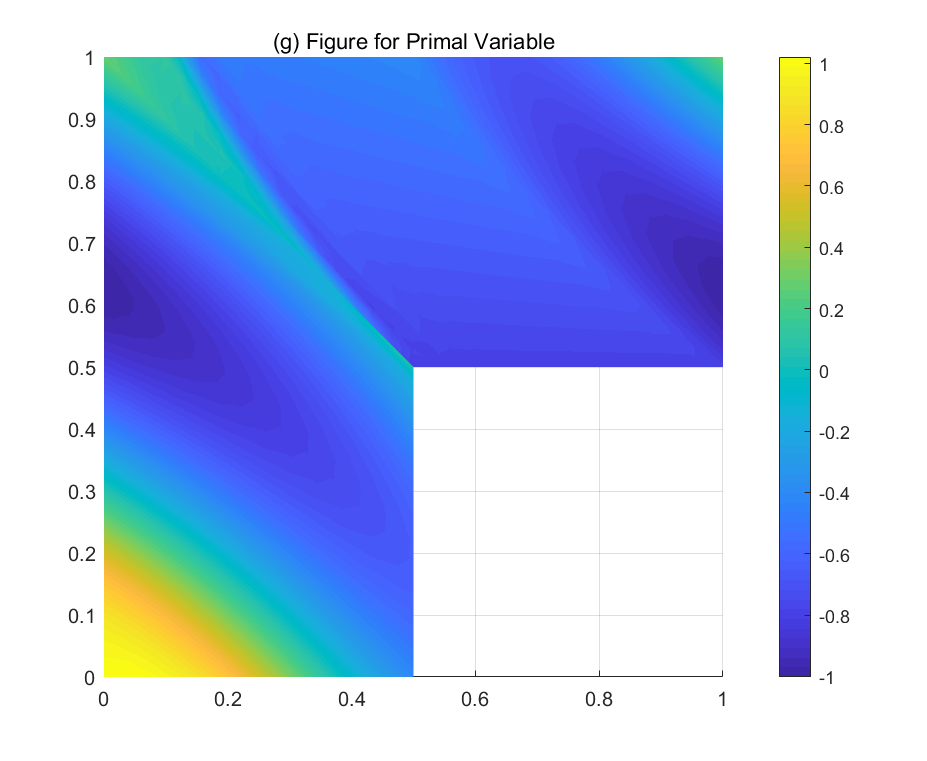}}
\resizebox{2.25in}{1.7in}{\includegraphics{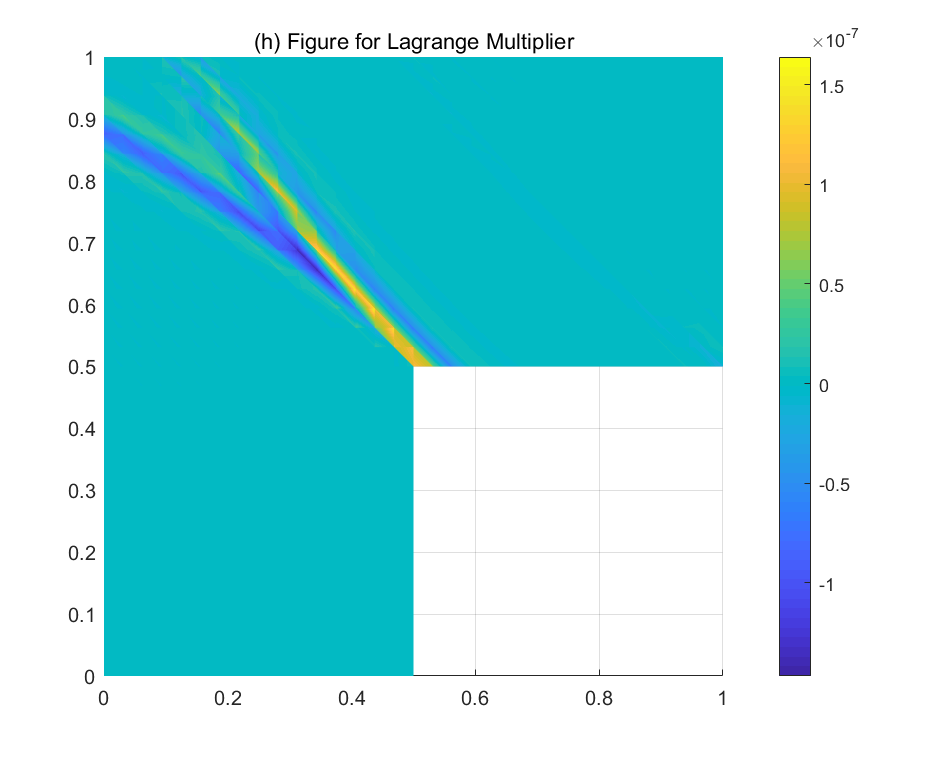}}
\end{tabular}
\caption{Contour plots for numerical solutions with $k=2$ and $j=k-1$ on $\Omega_2$, (e) the primal variable $u_h$ and (f) the Lagrange multiplier $\lambda_0$ for $f=0$, (g) the primal variable $u_h$ and (h) the Lagrange multiplier $\lambda_0$ for $f=1$, the coefficients $\{p,\rho,\tau\}=\{5,1e+12,0\}$.}\label{NE15}
\end{figure}

\end{document}